%% file: main.tex
\documentclass[11pt, a4paper]{amsart}

\usepackage[T1]{fontenc}
\usepackage[utf8]{inputenc}
\usepackage{lmodern}

\usepackage[american]{babel}
\usepackage[babel, final]{microtype}
\usepackage{amssymb}
\usepackage[all]{xy}

\usepackage{mathtools}

\usepackage{thmtools, thm-restate}

\usepackage{ifpdf}
\usepackage{comment}
\usepackage{multirow}

\usepackage{enumitem}

\usepackage[normalem]{ulem}

\usepackage[pdftex, dvipsnames]{xcolor}
\usepackage[pdftex, final]{graphicx}
\usepackage{caption}
\usepackage{subcaption}
\usepackage{pinlabel}
\usepackage[pdftex,%
    a4paper,%
    includehead,%
    includefoot,%
    nomarginpar,%
    lmargin=0.9in,%
    rmargin=0.9in,%
    tmargin=1in,%
    bmargin=1in,%
]{geometry}
\usepackage[pdftex,%
    final,%
    colorlinks=true,%
    linkcolor={red!45!black},%
    citecolor=NavyBlue,%
    filecolor=NavyBlue,%
    menucolor=NavyBlue,%
    urlcolor={blue!45!black},%
    bookmarks=true,%
    bookmarksdepth=3,%
    bookmarksnumbered=true,%
    bookmarksopen=true,%
    bookmarksopenlevel=2,%
]{hyperref}
\hypersetup{
    pdftitle={Unknotting numbers of 2-spheres in the 4-sphere},
    pdfauthor={Jason Joseph, Michael Klug, Benjamin Ruppik, Hannah Schwartz},
    pdfsubject={Unknotting numbers of 2-spheres in the 4-sphere},
    pdfkeywords={Unknotting numbers}
}
\usepackage{amsmath}
\usepackage{cleveref}

\usepackage{csquotes}
\usepackage[
	backend=bibtex,
	style=alphabetic,
	sorting=nyt,
	giveninits=true,
	maxnames=10,
	url=false,
	backref=true
]{biblatex}
\DefineBibliographyStrings{english}{
  backrefpage={$\uparrow$},
  backrefpages={$\uparrow$}
}

\DeclareFieldFormat[article,incollection,inproceedings]{title}{#1}
\DeclareFieldFormat[unpublished]{title}{\textit{#1}}

\usepackage{tikz}
\usetikzlibrary{tikzmark}
\usepackage{tikz-cd}
\usetikzlibrary{decorations.pathmorphing}

\usepackage{tikzsymbols}

\usepackage{listings}

\usepackage{float}
\usepackage{tabularx}
\usepackage[percent]{overpic}

\input{macros}
\input{commands}

\title{
    Unknotting numbers of $2$-spheres in the $4$-sphere
}

\author{Jason M. Joseph}
\address{Rice University, Houston, Texas 77005}
\email{\href{mailto:jason.joseph@rice.edu}{jason.joseph@rice.edu}}
\urladdr{\url{https://profiles.rice.edu/faculty/jason-joseph}}

\author{Michael R. Klug}
\address{University of Chicago, Chicago, Illinois 60637}
\email{\href{mailto:michaelklug@uchicago.edu}{michaelklug@uchicago.edu}}
\urladdr{\url{https://mathematics.uchicago.edu/people/profile/michael-klug/}}

\author{Benjamin M. Ruppik}
\address{Max-Planck-Institut f{\"ur} Mathematik, Bonn, Germany}
\email{\href{mailto:bruppik@mpim-bonn.mpg.de}{bruppik@mpim-bonn.mpg.de}}
\urladdr{\url{http://bruppik.de}}

\author{Hannah R. Schwartz}
\address{Princeton University, Princeton, New Jersey 08544}
\email{\href{mailto:hs25@princeton.edu}{hs25@princeton.edu}}
\urladdr{\url{https://www.math.princeton.edu/people/hannah-schwartz}}

\thanks{JJ, MK, BR \& HS were supported by the
Max Planck Institute for Mathematics in Bonn.}


\def\subjclassname{\textup{2020} Mathematics Subject Classification}
\expandafter\let\csname subjclassname@1991\endcsname=\subjclassname
\expandafter\let\csname subjclassname@2000\endcsname=\subjclassname
\subjclass{
    57K10, 
    57K40, 
    57K45, 
    57R42, 
    57R52 
    \hfill Date: \today
}

\begin{document}

\begin{abstract}
    We compare two naturally arising notions of ``unknotting number'' for $2$-spheres
    in the $4$-sphere: namely, the minimal number of $1$-handle stabilizations
    needed to obtain an unknotted surface,
    and the minimal number of Whitney moves required in a regular homotopy
    to the unknotted $2$-sphere.
    We refer to these invariants as the
    \emph{stabilization number}
    and the 
    \emph{Casson-Whitney number}
    of the sphere, respectively.
    Using both algebraic and geometric techniques,
    we show that the stabilization number is bounded above
    by one more than the Casson-Whitney number.
    We also provide explicit families of spheres for which these invariants are equal,
    as well as families for which they are distinct. 
    Furthermore, we give additional bounds for both invariants,
    concrete examples of their non-additivity,
    and applications to classical unknotting number
    of 1-knots.
\end{abstract}

\maketitle

\section{Introduction and Motivation}
\label{sec:motivation}

This paper compares and relates a slew of algebraic and geometric measures
of complexity of $2$-knots in the $4$-sphere.
What has traditionally been called the ``unknotting number''
of a $2$-knot $K \subset S^4$, which we call the \emph{stabilization number $\ust(K)$},
records the minimal number of stabilizations of $K$
required to obtain a smoothly embedded surface that bounds
a solid handlebody \cite{hosokawa1979numerical}.
This is analogous to the minimal number of $1$-dimensional stabilizations
(i.e.\ band attachments) of a $1$-knot needed to obtain an unlink.
This is bounded above by, but is \emph{not} in general equal to,
the classical unknotting number:
indeed there are many examples of low-crossing
knots for which this inequality is strict.

The classical unknotting number of a $1$-knot embedded in the $3$-sphere records the minimal number
of double points that occur during any regular homotopy to the unknot.
The analogue we consider in the $4$-dimensional setting
is the minimal number of Whitney moves needed
in a regular homotopy taking a 2-knot $K$ to the unknot
(double points are introduced/removed by a finger move/Whitney move). We call the minimal number of Whitney moves the \emph{Casson-Whitney number $\ufw(K)$} of the knot $K$, since techniques for manipulating finger moves (the inverse homotopy to a Whitney move) were pioneered by Casson \cite{casson1986three}. In \autoref{sec:inequality}, we use recent results of \cite{singh2019distances} to obtain the following relationship between the stabilization number and the Casson-Whitney number. 

\begin{restatable}{theorem}{plus_one}
    \label{thm:1} 
    For any $2$-knot $K$, $\ust(K) \leq \ufw(K) +1$. 
\end{restatable}

\smallskip

A careful manipulation of simple regular homotopies to the unknot
in \autoref{geo} also gives settings in which this inequality is always strict. 
 
\begin{restatable}{theorem}{equal_one}
    \label{thm:2} 
    Any $2$-knot $K$ with $\ufw(K)=1$ also has $\ust(K)=1$.  
\end{restatable}
 
\smallskip

Moreover, by considering the effect of finger moves and stabilizations on the fundamental group of the complement, we are able to find examples of $2$-knots for which equality of these unknotting invariants does \emph{not} hold. 

\begin{restatable}{theorem}{nonequal}
    \label{thm:3}
    There are infinitely many $2$-knots $K$ with $\ust(K)=1$ and $\ufw(K)=2$. 
\end{restatable}

\smallskip

In \autoref{geo}, we also give some special families of 2-knots in which
we can bound both the stabilization number and the Casson-Whitney number from above,
using explicit geometric constructions.
For instance, we find that the fusion number of a ribbon $2$-knot
is an upper bound for the Casson-Whitney number.

\begin{restatable}{theorem}{fusionbound}
    \label{thm:fusion_upper_bound}
    For a ribbon $2$-knot $K$,
    $\ufw(K) \le \fus(K)$.
\end{restatable}

The analogous result for the stabilization number $\ust(K) \le \fus(K)$ is due to Miyazaki \cite{miyazaki1986relationship}. 
In \autoref{sec:algebraic_lower_bounds}, we develop
the {\it algebraic Casson-Whitney number $\afw$},
a natural lower bound for $\ufw$,
and prove that for a pair of 2-knots,
both admitting a Fox coloring, the Casson-Whitney number
of their connected sum must be at least 2. A {\it Fox coloring} of a 2-knot is a surjection from the fundamental group of its complement onto a dihedral group, which sends meridians of the 2-knot to reflections. The {\it determinant} of a 2-knot is the evaluation of its Alexander ideal $\Delta(K)$ at $t=-1$, and as in the classical case a 2-knot $K$ has a $p$-coloring for prime $p$ if and only if $p$ divides its determinant $\Delta(K)|_{-1}$ \cite{joseph20190concordance}.
Thus, Casson-Whitney number one 2-knots cannot be factored
into a connected sum of two 2-knots, each with nontrivial determinant
(cf.\ the result of Scharlemann that unknotting number one 1-knots
are prime \cite{scharlemann1985unknotting}).

\begin{restatable}{theorem}{algadd}
    \label{thm:algadd} 
    Let $K_1,K_2$ be 2-knots with determinants
    $\Delta(K_i)|_{-1}\neq 1$.
    Then $\ufw(K_1\cs K_2)\geq 2$.
\end{restatable}

Miyazaki found 2-knots $K_1,K_2$ with $\ust(K_i)=1$ but $\ust(K_1\cs K_2)=1$
as well \cite{miyazaki1986relationship}.
Since his examples have nontrivial determinants,
these examples together with the above theorem imply \autoref{thm:3}.
The non-additivity of both $\ust$ and $\ufw$ is discussed in \autoref{sec:nonadd}, where we provide explicit families of $2$-knots for which additivity fails by an arbitrarily large amount.

\begin{restatable}{theorem}{nonadd}
    \label{thm:nonadd}
    For any positive $c, n \in \mathbb{N}$,
    there exist $2$-knots $K_1, \dots, K_n$ with
    \begin{align*}
        & \ust(K_i) = \ufw(K_i)  = c, \\
        & c \leq \ust(K_1\cs\cdots\cs K_n) \leq 2c, \text{ and }  \\
        & c \leq \ufw(K_1\cs\cdots\cs K_n)  \le 2c.
    \end{align*}
\end{restatable}

In the final section, we suggest some possible directions for further study.
Recently, the relationship between two similar invariants
$d_{\textrm{st}}$ and $\dsing$ was studied by
Singh \cite{singh2019distances}
(the invariant $\dsing$ already appeared as
$\mu_{\textup{sing}}$ in \cite{juhasz2018stabilization}).
His invariants record the minimal ``width'' of a sequence of stabilizations
and destabilizations of a regular homotopy,
meaning the maximum number of stabilizations or double points that occur \emph{simultaneously}.
The invariants we consider, on the other hand, record the minimal
``length'' of a sequence of stabilizations and destabilizations of a regular homotopy,
meaning the total number of stabilizations or double points that occur \emph{overall}.
Many of the geometric techniques used in our arguments are inspired
by those of Singh, as well as both Gabai \cite{dave:lightbulb}
and Schneiderman and Teichner \cite{schneiderman2019homotopy}.
A recent paper of Miller and Powell \cite{miller2019stabilization}
also studies the stabilization distance between arbitrary surfaces in $S^4$,
as well as the related relative setting of properly embedded surfaces in $B^4$. 

All manifolds and maps will be smooth unless stated otherwise.
All (ambient) manifolds will also be assumed to be connected,
and both manifolds and surfaces will be both compact and orientable.

\subsection*{Acknowledgements}
We greatly appreciate the support and advice given to us by
David Gay, Rob Kirby, Mark Powell, Arunima Ray, Rob Schneiderman, and Peter Teichner.
In addition, we would like to thank the entire topology community
at MPIM Bonn for providing such a welcoming and stimulating atmosphere
in which to collaborate and learn.
A special thank you to Rob, for lunch in Pisa,
and to Mark, Aru and Rob for their helpful edits.
We thank the anonymous reviewers for their careful reading
of our manuscript and their many 
insightful comments and suggestions.

\section{Background}
\label{sec:background}

For all definitions below, let $S$
be a smoothly immersed surface in $S^4$.
We use the shorthand
$\pi S \coloneqq \pi_1(S^4 - N(S), \ast)$
for the fundamental group of the complement
(of a neighborhood $N(S)$ of $S$),
where a basepoint $\ast$ is understood.
This section will mainly be spent analyzing the algebraic impact of the geometric operations we will be interested in.

\subsection{The stabilization number}

\begin{definition}
    \label{stab} 
    Suppose $\alpha$ is an arc with interior embedded in $S^4 -S$, whose endpoints lie on the surface $S$. The normal bundle of $\alpha$ in $S^4$ contains an embedded copy of $D^2 \times I$ intersecting $S$
    in exactly $D^2 \times \partial I$ such that the surface
    $(S - (D^2 \times \partial I)) \cup (\partial D^2 \times I)$\footnote{With corners smoothed, as in \autoref{fig:stabdiagrampdf}.}
    is orientable.
    This resulting surface is called the
    \bit{stabilization of $S$ along $\alpha$}.
    We call $\alpha$ the \bit{guiding arc for the stabilization}. 
\end{definition}

\fig{100}{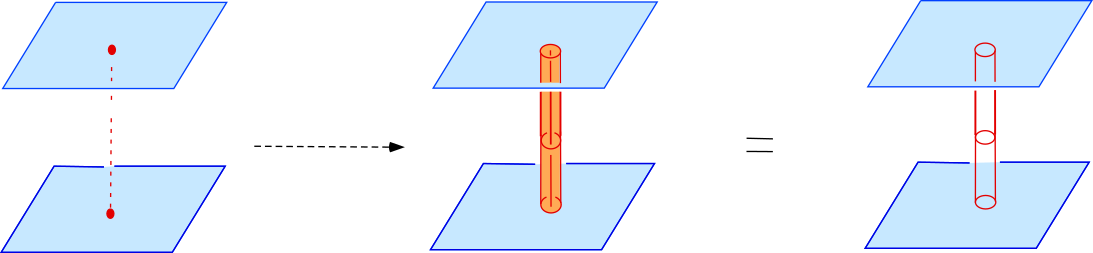}{
    \put(-420,6){$S$}
    \put(-381,47){$\alpha$}
    \put(-332,50){\small{Add $1$-handle}}
    \put(-135,50){$\partial$}
    \caption{The stabilization (right) of a surface $S$ along guiding arc $\alpha$.}
    \label{fig:stabdiagrampdf}
}

Observe that, as suggested by our terminology, the isotopy class of the stabilization depends only on the guiding arc $\alpha$ and not on the choice of sub-bundle $D^2 \times I$ (see \autoref{framings} for a similar discussion and Remark $5.3$ of \cite{dave:lightbulb} for more detail). Since guiding arcs with the same endpoints that are homotopic rel boundary are also isotopic rel boundary in dimension $4$, the stabilizations along these arcs are isotopic. 

\begin{definition}
    \label{defunknot} 
    A smoothly \bit{unknotted surface} in $S^4$ is a smoothly embedded surface
    of any genus that bounds a smoothly embedded solid handlebody.
\end{definition}

Indeed, since $S^4$ is simply-connected
the cores of the 1-handles of any pair of handlebodies of the same genus are isotopic.
This can be used to guide an isotopy to show that
there is a unique
unknotted surface of each genus.

Any closed orientable surface $K$ in the $4$-sphere is smoothly isotopic
to an unknotted surface after a finite number of stabilizations.
To see this, note that such a surface $K$ bounds
a smoothly embedded $3$-manifold $M \subset S^4$
called its \emph{Seifert solid}
which can be built as a handlebody from $K \times I$
by attaching $1$-handles to $K \times \{1\}$, followed by $2$ and $3$-handles.
Performing stabilizations to $K$ along the core arcs of the $1$-handles of $M$
gives a surface $K'$ that bounds the solid handlebody consisting of
the $2$ and $3$-handles of $M$, and so by definition is unknotted. 

\begin{definition}
    The \bit{stabilization number} $\ust(K)$ of a $2$-knot $K$
    is the minimal number of $1$-handle stabilizations needed to obtain an unknotted surface.
\end{definition}

\subsection{The Casson-Whitney number}

By Smale \cite[Theorem D]{smale} and Hirsch \cite[Theorem 8.3]{hirsch1959immersions}, 
embedded surfaces in a orientable $4$-manifold are homotopic
if and only if they are regularly homotopic,
i.e.\ homotopic through immersions.
Generically, there are only finitely many times during a regular homotopy
at which the immersed sphere is not self-transverse -- at these times,
double points of opposite sign are either introduced or cancelled. 

\begin{definition}\label{fmwm}
    The local model for the regular homotopy removing pairs
    of double points is called a \bit{Whitney move};
    this homotopy is supported in a regular neighborhood
    of a \bit{Whitney disk} $W$. 
    The inverse to this homotopy is called a \bit{finger move},
    which is supported in a regular neighborhood of an arc $\alpha$ whose endpoints lie on the surface, and whose interior is embedded in the complement.
    We call the arc $\alpha$ a \bit{guiding arc for the finger move}
    (the analog in this context of the guiding arc from \autoref{stab}).
    These two homotopies are depicted in \autoref{fwpdf}.
    Also labelled in the figure are the \bit{Whitney arcs} $\omega_1$ and $\omega_2$,
    whose union is the boundary of the Whitney disk $W$. Each Whitney arc connects a pair of double points along a sheet of the immersed surface. 
\end{definition} 

\fig{200}{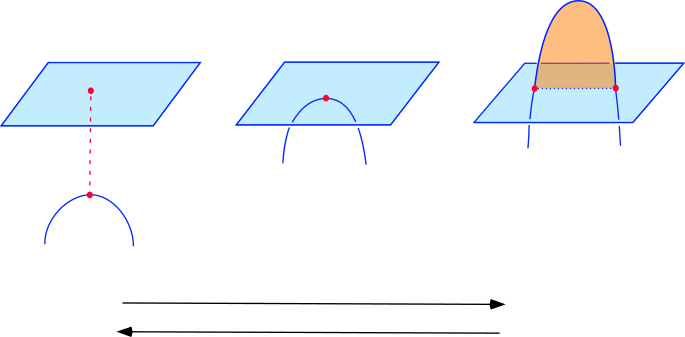}{
    \put(-70,173){$W$}
    \put(-348,105){$\alpha$}
    \put(-71,137){\small{$\omega_2$}}
    \put(-43,184){\small{$\omega_1$}}
    \put(-235,25){\small{Finger move}}
    \put(-235,-8){\small{Whitney move}}
    \caption{The local model of a finger move along the guiding arc $\alpha$, and Whitney move along the Whitney disk $W$ with boundary the union of the Whitney arcs $\omega_1$ and $\omega_2$.}
    \label{fwpdf}
}

\begin{remark} \label{framings}
To explicitly define a Whitney move in local coordinates requires that the normal disk bundle of the Whitney disk be framed ``compatibly'' with respect to its boundary on the immersed sphere; refer to \cite{freedman-quinn:4-manifolds} as well as Casson's lectures in \cite{casson1986three} for more details. Likewise, a framing of the normal $B^3$-bundle of the guiding arc compatible with the surface at its endpoints is needed to explicitly define a finger move. The choice of framing for the guiding arc will be suppressed, however, since (as with stabilizations) the resulting immersed surface up to ambient isotopy is independent of this choice of framing and depends only on the homotopy class of the arc itself, rel boundary
(see the discussion in \cite[Remark 5.3]{dave:lightbulb} for instance). 
\end{remark}

From now on, we shall always consider generic regular homotopies, in the sense that they are compositions of finger and Whitney moves as in \autoref{fmwm}. In fact, since the guiding arcs of the finger moves can be isotoped away from the Whitney disks in the ambient $4$-manifold, a deformation of the homotopy (without increasing the number of finger and Whitney moves) arranges for all of the finger moves to occur first, and simultaneously,
followed by all of the Whitney moves. A more detailed discussion of this folklore fact can be found in \cite[Section 4.1]{quinn1986isotopy}.
We will also always assume that our regular homotopies are of this form.

\begin{definition}
    The \bit{length} of a regular homotopy between surfaces
    is its total number of finger moves,
    or equivalently, Whitney moves.
    The \bit{Casson-Whitney number} $\ufw(K)$ of a $2$-knot $K$
    is the minimal length of any regular homotopy from $K$ to the unknot.  
\end{definition}

In general, finger moves (like stabilizations) depend on the choice of guiding arc up to homotopy rel boundary.  Namely, if two guiding arcs are homotopic and hence isotopic rel endpoints, then performing finger moves along these arcs results in immersions that are ambiently isotopic in $S^4$. In particular, it is critical to many of our arguments that all guiding arcs -- and hence finger moves -- are isotopic in the complement of the unknot $U$. 

\begin{definition} \label{stddef}
    We call the result of performing $n$ finger moves on the unknot $U$
    the \bit{standard immersed sphere} with $2n$ double points.
    Often, we reserve the use of $\Sigma$ to denote this immersion. 
\end{definition}

We later observe that there is indeed a unique standard immersed sphere for each $n$, up to ambient isotopy of $S^4$. 

\begin{definition}\label{disksdef}
After the finger moves and before the Whitney moves, any regular homotopy from a $2$-knot $K$ to the unknot $U$ restricts to the standard immersion $\Sigma$. 
\begin{equation*}
    \begin{tikzcd}[column sep=3cm]
        \text{2-knot } K
        \arrow[r, "\text{finger moves}", squiggly, shift left]
        &
        \text{standard immersion $\Sigma$}
        \arrow[l, "\text{Whitney moves}", squiggly, shift left]
        \arrow[r, "\text{Whitney moves}", squiggly, shift left]
        &
        \text{unknot } U
        \arrow[l, "\text{finger moves}", squiggly, shift left]
    \end{tikzcd}
\end{equation*}

\begin{figure}[!htbp]
    \centering
    \includegraphics[width=\linewidth]{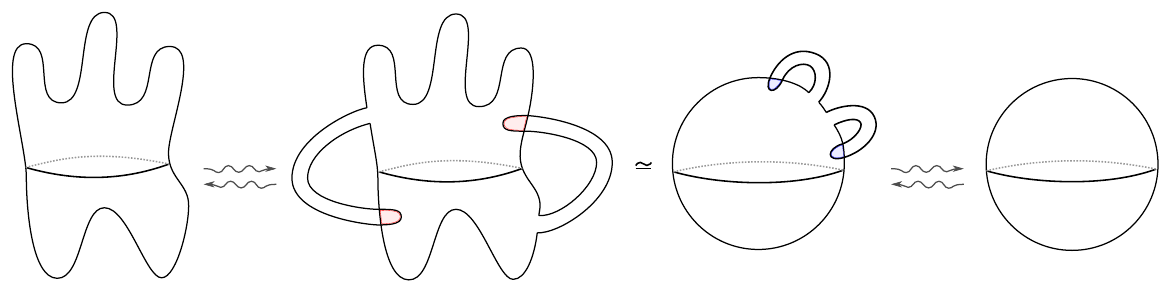}
    \put(-430,-5){$K$}
    \put(-43,-5){$U$}
     \put(-285,-5){$\Sigma$}
    \put(-166,-5){$\Sigma$}
    \put(-380,54){FM}
    \put(-380,30){WM}
    \put(-109,54){WM}
    \put(-109,30){FM}
    \caption{
        Decomposing a regular homotopy from a 2-knot $K$ to the unknot $U$. 
        The standard immersed sphere $\Sigma$ obtained
        after the finger moves (FM) and before Whitney moves (WM)
        on $K$ is drawn from two different perspectives
        (middle left and middle right) 
        to show the knotted and standard Whitney disks (red and blue respectively).
    }
    \label{fig:regular_homotopy_cartoon}
\end{figure}

Therefore, a regular homotopy 
from a 2-knot $K$ to the unknot $U$
is given by two collections
of Whitney disks that pair the double points
of the standard immersion: a set of \bit{standard}
Whitney disks leading to the unknot $U$,
and a set of \bit{knotted} Whitney disks leading to the knot $K$, as illustrated in \autoref{fig:regular_homotopy_cartoon}.
\end{definition}

\subsection{Fundamental group calculations}  \label{subsec:groupcalcs}

Below, we describe the effects of finger moves and stabilizations
on $\pi_1$ of the complement of a (possibly immersed) surface
$S = S_1 \cup \cdots \cup S_n$.
In particular, each move introduces one relation to $\pi_1$
as stated in the results below.
For detailed proofs, we refer to the original sources
\cite{casson1986three}, \cite{boyle1988classifying},
\cite{kirby2006topology}. 


Begin by picking a basepoint $\ast$ in the complement of the link $S$
and a basepoint $\ast_i \in S_i$ on each component of $S$.
For each $i$, fix an arc $\rho_i$ with interior in $S^4 - S$
connecting the basepoint $\ast$ to the basepoint $\ast_i$.

\fig{150}{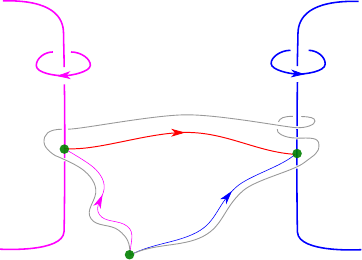}{
\put(-90,60){$\alpha$}
\put(-199,100){$m_{1}$}
\put(-24,100){$m_{2}$}
\put(-135,-6){$\ast$}
\put(-190,55){$\ast_1$}
\put(-25,55){$\ast_2$}
\put(-140,30){$\rho_1$}
\put(-100,30){$\rho_2$}
\put(-190,-5){$S_1$}
\put(-25,-5){$S_2$}
\caption{
    A choice of pushoff in gray, giving an element of $\pi S$ corresponding to the guiding arc $\alpha$.
    Also pictured are unbased meridians $m_i$
    for the components $S_i$.
} 
\label{fig:homotopic_stabilizing_tubes_schematic}}

\begin{definition}
    \label{meridian}
    A \bit{meridian} of $S$, and more specifically of the component $S_i$,
    is an element of $\pi S$ that can be represented by a simple closed curve
    $\gamma \colon S^1\hookrightarrow S^4 - S$ bounding a disk in $S^4$
    that transversely intersects $S_i$ in a single point.
\end{definition}

An orientation of $S$ and the ambient space
induces a positive orientation on the meridian.
The set of positively oriented
meridians of a connected component
of a knotted surface forms a conjugacy class of the
fundamental group of its complement.
That is, $x$ is a meridian of $S_i$ if and only if
$x^w \coloneqq w^{-1}xw$ is as well,
for any $w\in\pi S$. 
If $S$ is connected,
this element $w$ may be chosen to lie in the commutator subgroup
$(\pi S)^\prime$,
a fact which we exploit in \autoref{sec:algebraic_lower_bounds}.

\begin{definition}
    \label{defguidingarc}
    Let $\alpha$ be an arc with interior embedded away from $S$,
    connecting $\ast_i$ to $\ast_j$ for (possibly equal) indices $i,j$.
    We call this a \bit{guiding arc for $S$}, as we did without specifying base points in Definitions \ref{stab} and \ref{fmwm}.
    Each push-off of the loop $\rho_i \alpha \rho_j^{-1}$ into the complement
    of the surface $S$ gives an element $g \in \pi S$
    that is said to \bit{correspond} to the arc $\alpha$,
    see \autoref{fig:homotopic_stabilizing_tubes_schematic}.
    Note that the element $g$ is well-defined (i.e.\ independent of the push-off)
    up to left multiplication by meridians of $S_i$
    and right multiplication by meridians of $S_j$. 
\end{definition}

From now on, we will always assume that the guiding arcs used for both stabilizations and finger moves are of this form, i.e. connecting a basepoint $\ast_i \in S_i$ to a basepoint $\ast_j \in S_j$ for some (possibly equal) indices $i, j$. We often refer to arcs corresponding to the identity element, as well as stabilizations and finger moves done along such a guiding arc, as ``trivial''. 

\begin{remark}
    \label{rem:boundarytwist}
    Let $\alpha$ and $\beta$ be guiding arcs for $S$ with the same endpoints
    $\ast_i$ on $S_i$ and $\ast_j$ on $S_j$.
    Suppose that $\alpha$ corresponds to an element $g \in \pi S$
    and $\beta$ corresponds to $m_1^{n_1} g m_2^{n_2}$
    for some $n_1, n_2 \in \mathbb{Z}$ and meridians $m_i, m_j$ to $S_i, S_j$ respectively.
    Then, the guiding arcs $\alpha$ and $\beta$ are isotopic in the complement of $S$
    rel boundary via a sequence of ``boundary twists''
    as pictured in \autoref{fig:boundary_twist}.
    It follows that the surfaces obtained by either stabilizing
    or performing finger moves along these arcs are ambiently isotopic.
\end{remark}

\fig{150}{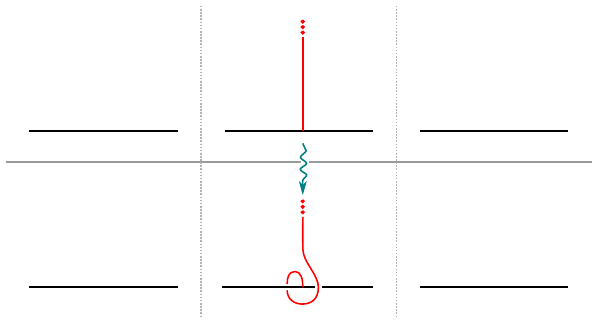}{
    \put(-260,80){$S_i$}
    \put(-170,80){$S_i$}
    \put(-80,80){$S_i$}
    \put(-130,120){\textcolor{red}{$\alpha$}}
    \put(-260,7){$S_i$}
    \put(-170,7){$S_i$}
    \put(-80,7){$S_i$}
    \put(-130,40){\textcolor{red}{$\beta$}}
    \caption{
        The guiding arcs from \autoref{rem:boundarytwist},
        before (top) and after (bottom) a boundary twist.
    }
    \label{fig:boundary_twist}
}

In particular, it follows from \autoref{rem:boundarytwist}
that all guiding arcs for the unknot $U$ are isotopic,
since $\pi U \cong \mathbb{Z}$.
So for any $n>0$, there is a unique surface (up to isotopy) resulting
from $n$ stabilizations of $U$ -- namely,
the genus $n$ unknotted surface, as in \autoref{defunknot}.
Similarly, the immersed sphere
resulting from $n$ finger moves on $U$ is ambiently isotopic
to the standard immersion with $2n$ double points, as in \autoref{stddef}.

\begin{lemma}[Stabilization relation]
    \label{lem:pi_1_handle}
    Let $\alpha_1, \dots, \alpha_k$ be disjointly embedded guiding arcs
    along which stabilizing $S$ gives the surface $S'$.
    Then
    \begin{equation*}
        \pi (S') \cong 
        \bigslant{ \pi S }{ \langle \langle g_i^{-1} a_i g_i b_i^{-1}  \rangle \rangle }
    \end{equation*}
    where $a_i,b_i$ are meridians to the components of $S$ containing the endpoints of $\alpha_i$ (as in \autoref{meridian}), and the element $g_i$ corresponds to $\alpha_i$ (as in \autoref{defguidingarc}). 
\end{lemma}

\begin{figure}[ht]
    \begin{subfigure}{.495\textwidth}
      \centering
      \includegraphics[width=.95\linewidth]{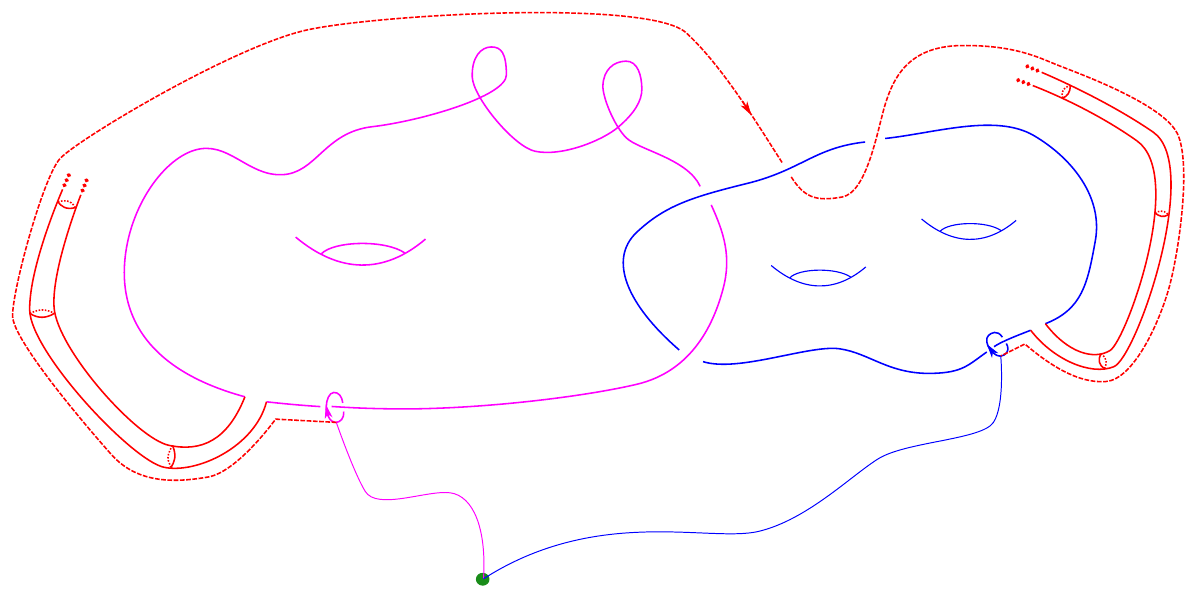}  
      \caption{Stabilization}
      \label{fig:pi_1_handle_stabilization}
    \end{subfigure}
    \begin{subfigure}{.495\textwidth}
      \centering
      \begin{overpic}[width=0.95\textwidth]{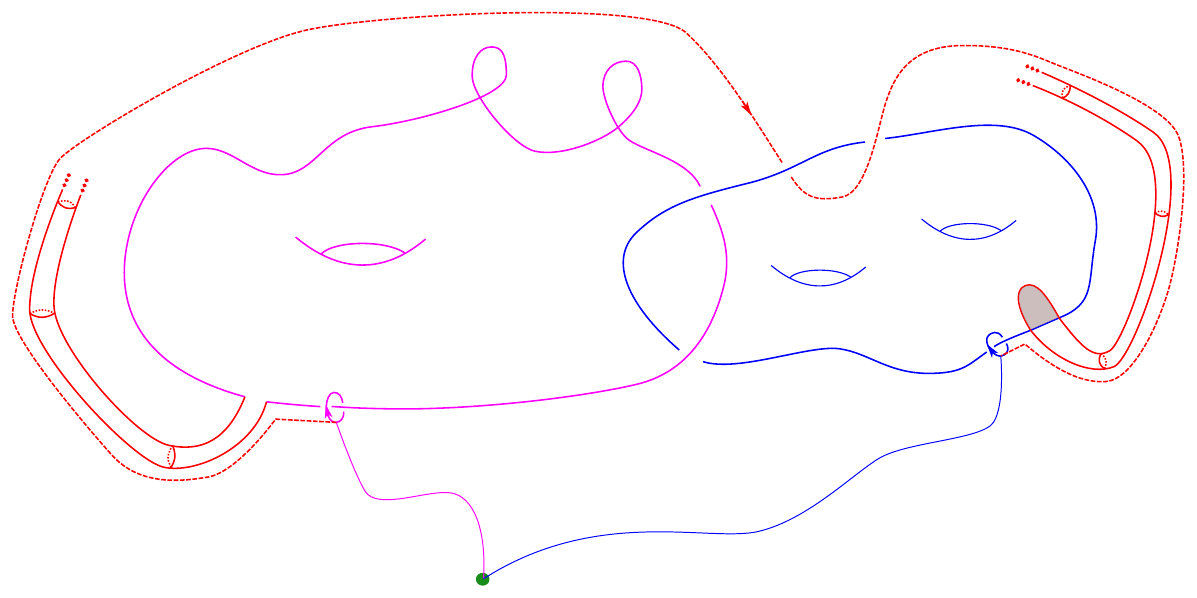}
      \end{overpic}

      \caption{Finger move \label{fig:pi_1_finger_move}} 
    \end{subfigure}
\caption{
    Illustrating \autoref{lem:pi_1_handle}
    and \autoref{lem:pi_1_finger_move}, with the
    meridian $a_{i}$ in pink and $b_{i}$ in blue.
    }
\label{fig:fig}
\end{figure}

Refer to \autoref{fig:pi_1_handle_stabilization}
for a schematic of the set-up in \autoref{lem:pi_1_handle}.
Note that each $g_i^{-1} a_i g_i$ is also a meridian; hence the relation introduced by stabilizing can also be thought of as one which simply identifies two meridians. We make the following definitions for $n$-knots, because we will think about them in reference to 1-knots as well as 2-knots.

\begin{definition} \label{def:algstab} 
    Let $K$ be an $n$-knot.
    The minimal number of relations of the form $x=y$,
    where $x,y$ are meridians of $K$, which abelianize the knot group
    is called the \bit{algebraic stabilization number} $\asta(K)$ of $K$.
\end{definition}


\begin{lemma}[Finger move relation]
    \label{lem:pi_1_finger_move}
    Suppose that $S'$ is the result of performing finger moves on $S$ along disjointly embedded guiding arcs $\alpha_1, \dots, \alpha_k$. Then, 
    \begin{equation*}
        \pi(S') \cong
        \bigslant{ \pi S }{ \langle \langle [a_i,g_i^{-1} b_i g_i] \rangle \rangle }
    \end{equation*}
   where $a_i,b_i$ are meridians to the components of $S$ containing the endpoints of $\alpha_i$ (as in \autoref{meridian}), and the element $g_i$ corresponds to $\alpha_i$ (as in \autoref{defguidingarc}).
\end{lemma}

\autoref{fig:pi_1_finger_move} gives a schematic
of the set-up in \autoref{lem:pi_1_finger_move}.
Note that while the stabilization relation identifies two meridians,
a finger move relation can only make them commute.
This discrepancy leads to our result in \autoref{sec:algebraic_lower_bounds}
that the stabilization and Casson-Whitney numbers are not equal in general.

\begin{definition}
    \label{def:algcw}
    Let $K$ be an $n$-knot.
    The minimal number of relations of the form $xy=yx$
    which abelianize the knot group,
    where $x, y$ are meridians of $K$,
    is called the \bit{algebraic Casson-Whitney number} $\afw(K)$ of $K$.
\end{definition}

This minimum gives an algebraic lower bound for $\ufw(K)$,
since a regular homotopy from a 2-knot $K$ to the unknot
starts with a sequence of finger moves on $K$ to the
standard immersion $\Sigma$ with $\pi_1(S^4- \Sigma) \cong \mathbb{Z}$;
thus the corresponding finger move relations abelianize $\pi  K$.

\begin{figure}[ht]
    \begin{subfigure}{.45\textwidth}
      \centering
      
      \begin{overpic}[width=0.95\textwidth]{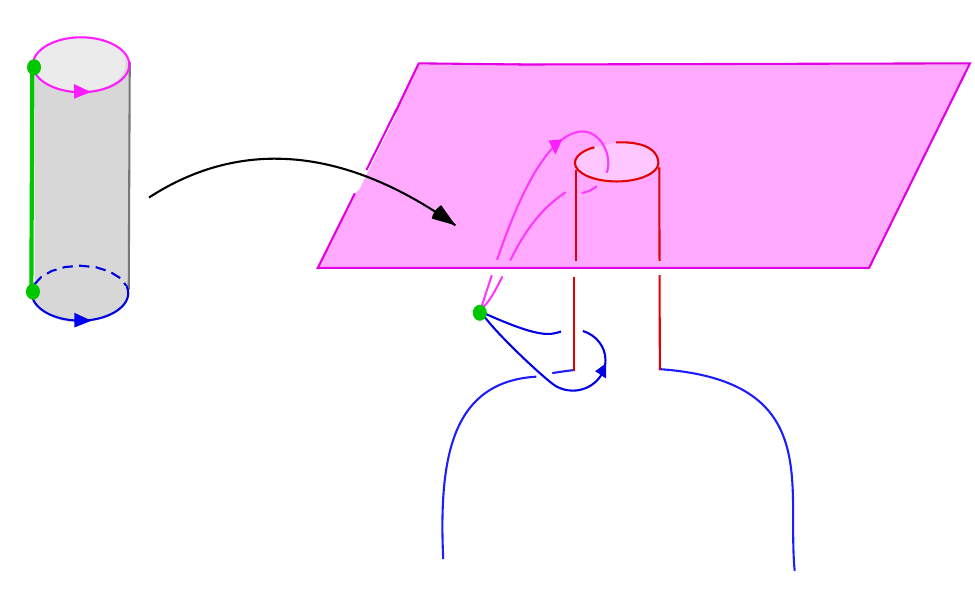}
      \end{overpic}
      \caption{Stabilization relation \label{fig:stabrelation}} 
    \end{subfigure}
    \begin{subfigure}{.45\textwidth}
      \centering
      \begin{overpic}[width=0.95\textwidth]{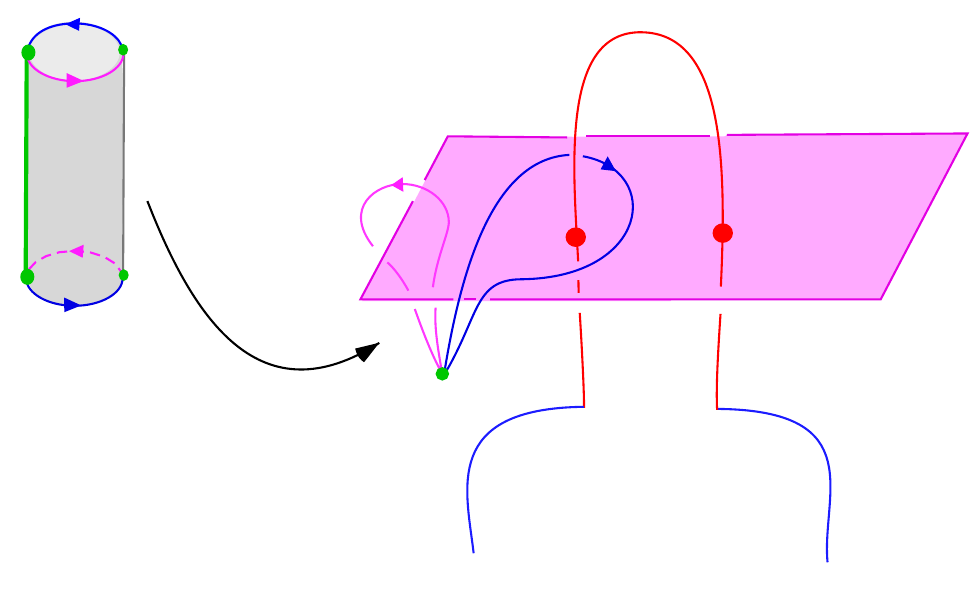}
      \end{overpic}
      \caption{Finger move relation  \label{fig:fmrelation}} 
    \end{subfigure}
\caption{
    In both figures (A) and (B), oriented meridians $x$ and $y$ to the surface are drawn in pink and blue, and the basepoint in green. The gray annuli (immersed in the complement of the surface) are null-homotopies giving the algebraic relations from \autoref{lem:pi_1_handle} and \autoref{lem:pi_1_finger_move}. 
    On the right, the image of the grey annulus is exactly the Clifford torus around the double point, illustrating that the commutator relation $[x, y] = 1$ holds.}
\label{fig:1-handle_alg}
\end{figure}

We summarize the results of this section in the following proposition.
To our knowledge, these are the sharpest algebraic lower bounds for the unknotting numbers.

\begin{proposition}
    \label{basic}
    For any 2-knot $K$,
    $\asta(K)\leq\ust(K)$ and $\afw(K)\leq\ufw(K)$.
\end{proposition}

\autoref{tab:invariants} gives a glossary of the main invariants
that will be referred to. 
\begin{center}
\begin{table}[H]
    \begin{tabularx}{1.05\textwidth}{ |c|l|X| }
    \hline
    $\pi K$ &
    knot group/surface group &
        fundamental group of knot complement $S^4 - K$ \\
    $\mu(K)$ & meridional rank of $K$ & 
        minimal number of meridians which generate $\pi K$ \\ 
    $m(K)$ & Nakanishi index &
        minimal size of generating set of Alexander module of $K$ \\ 
    $a(K)$ & Ma-Qiu index &
        minimal size of normal generating set of commutator subgroup $(\pi K)'$ \\
    $\ust(K)$ & stabilization number &
        minimal number of 1-handle stabilizations needed to obtain
        an unknotted surface from $K$ \\
    $\ufw(K)$ & Casson-Whitney number &
        minimal number of Whitney moves in a regular homotopy from $K$ to the unknot \\
    $\asta(K)$ & algebr.\ stabilization number & 
        minimal number of 1-handle stabilizations on $K$ needed to obtain
        a surface with group $\mathbb{Z}$ \\
    $\afw(K)$ & algebr.\ Casson-Whitney number &
        minimal number of finger moves on $K$ needed to obtain an
        immersed 2-knot with group $\mathbb{Z}$ \\
    $\fus(K)$ & fusion number of a ribbon knot &
        minimal number of fusion tubes in a ribbon presentation for $K$ \\
    \hline
    \end{tabularx}
    \caption{Overview of main invariants for a 2-knot $K$.}
    \label{tab:invariants}
\end{table}
\end{center}

\section{Relating the stabilization and Casson-Whitney numbers} \label{sec:inequality} 

Fix a $2$-knot $K \subset S^4$
and let $U \subset S^4$ denote the unknotted $2$-sphere. We begin by introducing some terminology needed only in this section.

\begin{definition}\label{tubedsurface} 
    Given an immersed surface $\Sigma$ in $S^4$ with algebraically zero double points,
    and any choice of disjointly embedded arcs on $\Sigma$ pairing double points of opposite sign,
    there is an \bit{associated tubed surface} obtained by tubing the double points
    along these arcs as in \autoref{fig:surface_tubing_definition}.
    Note that this surface is oriented, since the endpoints of each arc are double points of opposite orientation, and that a priori the smooth (and even topological)
    isotopy class of the resulting surface depends on the arcs along which the tubing is done. 
\end{definition}

\begin{figure}[ht]
    \begin{subfigure}{.495\textwidth}
      \centering
      \begin{overpic}[width=0.95\textwidth]{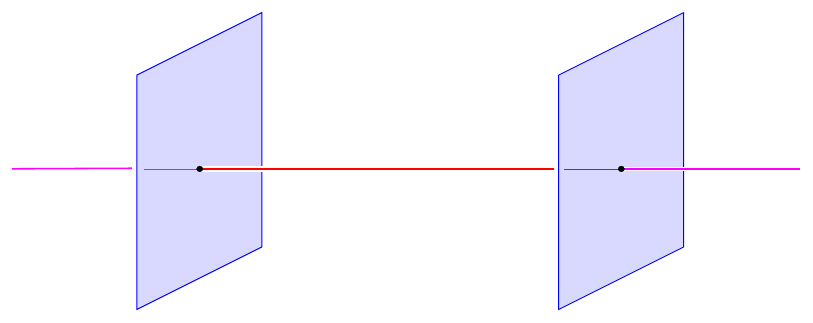}
        \put(50,20){\color{red} $\alpha$}
        \label{fig:surface_tubing_start}
      \end{overpic}
      \caption{
            Double points and connecting arc.
        }
    \end{subfigure}
    \begin{subfigure}{.495\textwidth}
      \centering
      \begin{overpic}[width=0.95\textwidth]{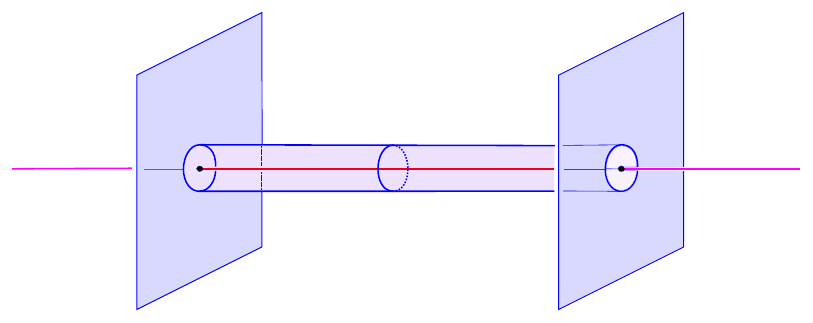}
      \put(34,25){\tiny linking annulus of {\color{red} $\alpha$}}
        \label{fig:surface_tubing_linking_annulus}
      \end{overpic}
      \caption{
            After the tubing.
        }
    \end{subfigure}
\caption{
    Tubing an immersed surface along an arc $\alpha$ (red) that connects oppositely signed double points. Sheets of the surface are drawn in pink and blue; the pink sheet is an arc persisting
    into the past and future. To tube the double points together, remove a disk in the blue sheet around each double point and add the linking annulus of the guiding arc $\alpha$ as shown on the right. 
    }
\label{fig:surface_tubing_definition}
\end{figure}

Although we choose not to define it rigorously here, the procedure of ``tubing'' employed in the definition above is described in Remark $5.3$
of \cite{dave:lightbulb}, as well as in Definition $2.6$ of \cite{singh2019distances}. Indeed, isotopies between associated tubed surfaces are the focus of both papers. To ensure that the associated tubed surfaces that arise in our discussions are isotopic, we will be especially interested in regular homotopies of the following type. 

\begin{definition}
    A regular homotopy of length $n$ from $K$ to $U$ in $S^4$ is called \bit{arc-standard} if its standard Whitney disks $W_1, \dots, W_n$ and knotted Whitney disks $V_1, \dots, V_n$ have at least one Whitney arc in common for each $i$. 
\end{definition}

\begin{remark}
    \label{fq}
    It is unknown whether or not every $2$-knot in $S^4$ admits an arc-standard regular homotopy to the unknot, let alone one of minimal length. There are many \emph{non-simply connected} $4$-manifolds containing pairs of $2$-spheres between which there is no analog of an arc-standard homotopy.
    For instance, any pair of spheres related by such a homotopy
    must have vanishing Freedman-Quinn
    invariant\footnote{This concordance invariant was defined by Freedman
        and Quinn \cite{freedman-quinn:4-manifolds} in the '$90$s,
        and later corrected by Stong \cite{stong}.
        Schneiderman and Teichner give a nice exposition
        in  \cite{schneiderman2019homotopy}.
    } since in this case all double curves of the trace of the homotopy are trivially double covered. 
    However, there are many instances where this does not hold -- see \cite{schneiderman2019homotopy}, \cite{dave:lightbulb},
    or \cite{hannah} for example.
\end{remark}

With this terminology in place, we state our first result. Although this fact is implied by Singh's proof of Theorem $1.4$ in \cite{singh2019distances}, we state and prove it here in our setting. 

\begin{proposition}
    \label{std}  
    If there is a length $n$ arc-standard homotopy from $K$ to $U$,
    then $K$ can be unknotted with $n$ stabilizations.
\end{proposition}
    
\fig{200}{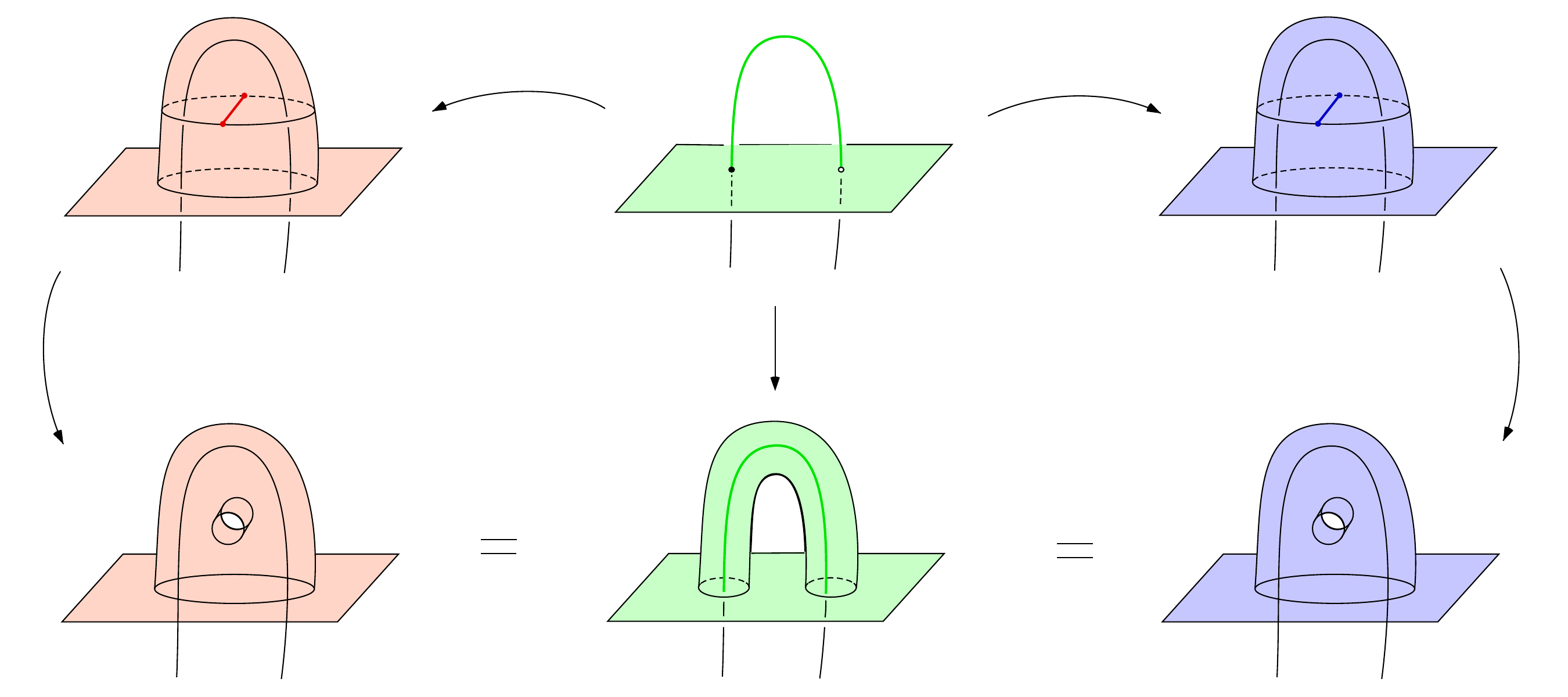}{
\put(-231,122){$\Sigma$}
\put(-360,0){$F_K$}
\put(-110,0){$F_U$}
\put(-360,120){$K$}
\put(-110,120){$U$}
\put(-175,188){\small{Whitney move}}
\put(-335,189){\small{Whitney move}}
\put(-165,178){\small{along $W_i$}}
\put(-325,179){\small{along $V_i$}}
\put(-475,98){\small{Attach}}
\put(-480,88){\small{$1$-handle}}
\put(-5,98){\small{Attach}}
\put(-10,88){\small{$1$-handle}}
\put(-215,100){\small{Tube}}
\put(-217,90){\small{along $\alpha_i$}}
\put(-207,180){$\alpha_i$}
\caption{
    ``Standard pictures'' of the spheres $U$ and $K$ (top right and left)
    given by Whitney moves on the immersion $\Sigma$,
    and their isotopic stabilizations $F_U$ and $F_K$ along the
    red or blue guiding arcs (lower right and left).
    Note that although the local models of $U$ and $K$
    after the Whitney moves look identical,
    the interiors of the Whitney disks $W_i$ and $V_i$,
    and hence these portions of $U$ and $K$, may be embedded very differently in $S^4$.
}
\label{stabpdf}
}

\begin{proof}
Such a regular homotopy is given by a set of standard Whitney disks $W_1, \dots, W_n$ and knotted Whitney disks $V_1, \dots, V_n$ for the standard immersion $\Sigma$ with $2n$ double points, as in \autoref{stddef} and \autoref{disksdef}.
Since the regular homotopy is arc-standard, 
by definition for each $i$ the standard and knotted Whitney disks
have at least one common Whitney arc $\alpha_i$. 

The end of the Whitney homotopy for each Whitney disk $W_i$ and $V_i$ gives a ``local model'' of the resulting embedded $2$-sphere, as illustrated in the top left and right of \autoref{stabpdf}. Stabilizing $K$ along guiding arcs connecting the sheets of $K$ parallel to each knotted Whitney disk gives an embedded genus $n$ surface $F_K$ shown on the bottom left of \autoref{stabpdf}. Likewise, stabilizing the unknot $U$ along guiding arcs connecting the sheets of $U$ parallel to each standard Whitney disk
gives a genus $n$ standard surface $F_U$ shown on the bottom right of \autoref{stabpdf}. 
Both $F_K$ and $F_U$ are isotopic to the associated tubed surface $\Sigma$ stabilized along the common Whitney arcs $\alpha_1, \dots, \alpha_n$ depicted on the bottom center of \autoref{stabpdf} (the tube along the Whitney arc $\alpha_i$ is shown in green).
Since the surface $F_U$ is a stabilization of the unknot,
it follows that $F_U$ and hence $F_K$ is unknotted.
\end{proof} 

\fig{180}{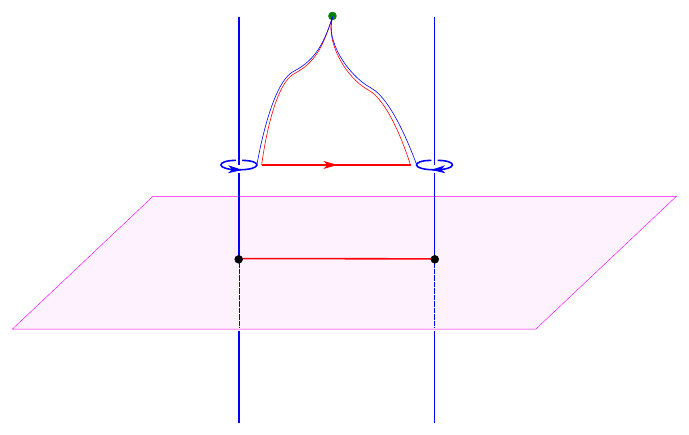}{
    \put(-205,110){{\textcolor{blue}{$a_i$}}}
    \put(-95,110){{\textcolor{blue}{$b_i$}}}
    \put(-150,120){{\textcolor{red}{$g_i$}}}
    \put(-200,65){{$p_i$}}
    \put(-100,65){{$q_i$}}
    \put(-150,65){{\textcolor{red}{$\alpha_i$}}}
    \caption{
        The initial situation in \autoref{lem:pi_1_associated_tubed}
        before the tubing, where the horizontal pink sheet 
        (with the red arc $\alpha_i$) lives completely in the present,
        while the vertical blue sheet is spread out in time.
        Also pictured are two meridians
        $a_{i}, b_{i}$ to the blue sheet,
        and the group element
        $g_{i}$ associated to the arc $\alpha_{i}$.
    }
    \label{fig:surface_tubing_movie_start}}


\fig{190}{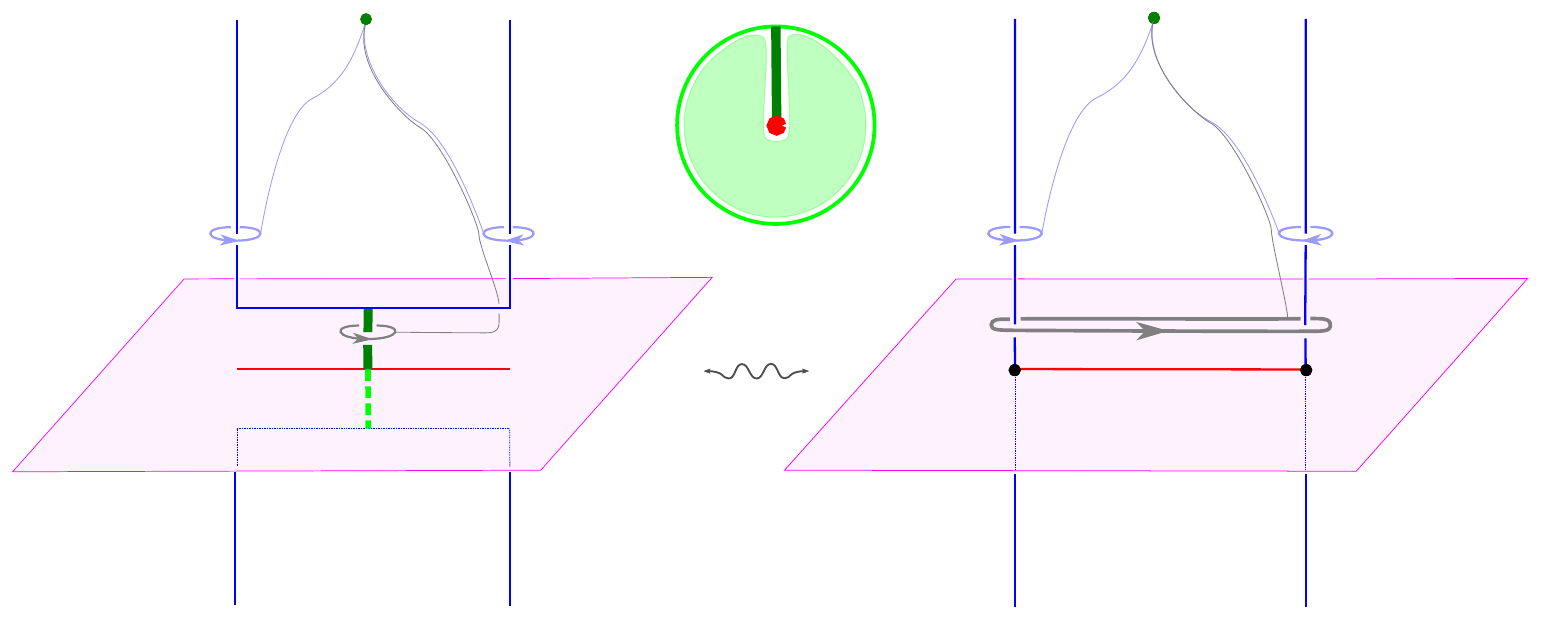}{
    \put(-200,160){{$D_{i}$}}
    \put(-240,130){{$D'_{i}$}}
    \definecolor{darkspringgreen}{rgb}{0.09, 0.45, 0.27}
    \put(-230,160){{$\textcolor{darkspringgreen}{\gamma_{i}}$}}
    \put(-365,99){{$\textcolor{darkspringgreen}{\gamma_{i}}$}}
    \put(-418, 86){{\tiny \textcolor{gray}{normal circle}}}
    \caption{The figure depicts a  homotopy equivalence between the complements
        of $(S' \cup \gamma_{i}) - D'_{i}$ (left)
        and $S$ (right), where we keep
        track of where the gray circle normal
        to the disk $D'_{i}$ is taken. }
    \label{fig:surface_tubing_homotopy_equivalence}
}

The following $\pi_1$ calculation is used in the proof of
\autoref{thm:1} and is very similar to Casson's proof of
\autoref{lem:pi_1_finger_move}, see \cite{casson1986three}. Before stating the lemma, we establish some necessary notation. For any $4$-manifold $X$, suppose that $S \subset X$ is an immersed surface with positive and negative double points $p_1, \dots, p_n$ and $q_1, \dots, q_n$, respectively. For each $i$, let $\alpha_i \subset S$ be an embedded arc connecting $p_i$ to $q_i$, and let $a_i, b_i \in \pi S$ be positively oriented meridians for the sheets of the double points that do \emph{not} contain $\alpha_i$. As illustrated in \autoref{fig:surface_tubing_movie_start}, the meridian $a_i$ is constructed by running along $\rho_i$, around the boundary of a disk normal to $S$, and back along $\rho_i^{-1}$ to the base point $\ast$ of $\pi_1 S$. The meridian $b_i$ is constructed analogously, but using the path $\eta_i$. 

The arc $\alpha_i \subset S$ corresponds to an element $ g_i \in \pi S$ given by the composition of paths $\rho_i \alpha_i' \eta_i^{-1}$,
where $\alpha_i'$ is a push-off of the arc $\alpha_i$
into the normal disk bundle of $S$ -- this element is hence well-defined
only up to twists around the normal disk bundle of $S$ restricted to $\alpha_i$.
However, as this indeterminacy does not affect the fundamental group calculations
below\footnote{This follows since the Clifford tori around the double points $p_i$ and $q_i$ allow the meridians $a_i$ and $b_i$ to commute with twists around $\alpha$.},
we suppress it from notation.  

\begin{lemma}[Tubing relation]
    \label{lem:pi_1_associated_tubed}
    Let $S \subset X^4$ be an immersed surface whose associated tubed surface $S'$  
    is constructed by tubing together oppositely signed double points $p_i$ and $q_i$ along arcs $\alpha_i \subset S$, as in \autoref{tubedsurface}. Then,
    \begin{equation*}
        \pi(S') \cong 
        \bigslant{ \pi S }{ \langle \langle g_i^{-1}a_i g_i  b_i^{-1}  \rangle \rangle }
    \end{equation*}
     for elements $a_i,b_i,g_i \in \pi S$ defined both in the paragraph above and illustrated in \autoref{fig:surface_tubing_movie_start}. 
\end{lemma}

\begin{proof}
    For each $i$, consider a disk $D_i$ normal to $S$ at an interior point of the arc $\alpha_i$, as in \autoref{fig:surface_tubing_homotopy_equivalence}. The intersection of this disk with the tubed surface $S'$ then consists of $\partial D_i$, together with the point where $\alpha_i$ intersects $D_i$. Notice that the complement of $S' \cup D_1 \cup \dots \cup D_n$ is homotopy equivalent to the complement of the immersion $S$. So, to compare $\pi(S')$ and $\pi S$, we remove the regular neighborhoods of each disk $D_i$ from $X- S'$ to obtain $X-S$ in two stages: first we delete neighborhoods of embedded arcs $\gamma_i \subset D_i$ connecting $\alpha_i$ to $\partial D_i$, and then delete neighborhoods of the remaining disks $D_i' \subset D_i$ whose boundary circles run around $\partial D_i$ and then forward and back along $\gamma_i$, as in \autoref{fig:surface_tubing_homotopy_equivalence}.  
    
    Let $S_\gamma'$ denote the union
    $S' \cup  N(\gamma_1) \cup \dots \cup N(\gamma_n)$.
    Note that $\pi(S') \cong \pi_{1}(S^4 - S_\gamma')$
    since each $\gamma_i$ has codimension three.
    The complement $X-S$ is obtained from $X-S'_\gamma$ by removing regular neighborhoods of the disks $D_i'$. Dually, this implies that the complement $X-S'_\gamma$
    is obtained from $X-S$ by attaching $n$ many $2$-handles to $X-S$ along the boundaries of disks normal to $D'_i$ in the complement of
    $S'_\gamma \cup N(D_1') \cup \dots \cup N(D'_n)$,
    which is diffeomorphic to $X-S$. Therefore, $\pi(S')$ is obtained from $\pi S$ by adding $n$ relators -- namely the boundaries of these 2-handles.
    The boundaries of these 2-handles are exactly the small gray circles as show in \autoref{fig:surface_tubing_homotopy_equivalence}.
    We make these into elements of $\pi S$ by pre- and post-composing
    the circles with the gray arc from the basepoint as shown in
    \autoref{fig:surface_tubing_homotopy_equivalence}.
    These elements are then exactly the elements
    $g_i^{-1}a_ig_ib_i^{-1} \in \pi S$
    as is shown by the homotopy-equivalence in 
    \autoref{fig:surface_tubing_homotopy_equivalence}.
\end{proof}

We conclude the section by proving \autoref{thm:1} and \autoref{thm:2}.
Note that Singh gives the analog of \autoref{thm:1}
for his related invariants
$d_{\textrm{st}}$ and $\dsing$
in \cite[Theorem $1.4$]{singh2019distances}, and in fact,
our proof of \autoref{thm:1} relies on Lemmas $3.1$ and $4.1$ of his paper.
We are unable to provide any examples in which the $+1$ term is necessary,
and so leave this as a question in \autoref{questions};
this question is also left open in Singh's setting.  

\begin{proof}[Proof of \autoref{thm:1}]
    Take a length $\ufw(K)$ regular homotopy from $K$ to $U$, with associated tubed surfaces $F_U$ and $F_K$ constructed as in the proof of \autoref{std}. Note that since the homotopy is not necessarily arc-standard, the double points of $\Sigma$ are tubed together along different arcs; thus it unclear whether or not $F_U$ and $F_K$ are isotopic. In his proof of \cite[Theorem $1.4$]{singh2019distances}, Singh produces a sequence of tubed surfaces $T_1, \dots, T_m$ for the standard immersed sphere $\Sigma$ such that $T_1 = F_U$, $T_m = F_K$, and such that each consecutive pair $T_i$ and $T_{i+1}$ become isotopic after a single stabilization.

    Note that $\pi  T_i \cong \mathbb{Z}$  for each $i$. This follows from \autoref{lem:pi_1_associated_tubed}, since each $T_i$ is an associated tubed surface for the standard immersion $\Sigma$ with $\pi \Sigma \cong \ZZ$. Thus, any stabilization of $T_i$ is isotopic to the stabilization done along the trivial guiding arc. This, combined with the fact that $T_i$ and $T_{i+1}$ become isotopic after a single stabilization, implies that the trivial stabilizations of the tubed surfaces $T_1, \dots, T_m$ are all pairwise isotopic; in particular $F_U$ and $F_K$ become isotopic after a single stabilization.
\end{proof}

\begin{remark}
\label{rmk:middle_level_group}
Note that in the proof of \autoref{thm:1} above,
it is critical that each ``intermediate'' tubed surface $T_i$ has $\pi  T_i \cong \mathbb{Z}$. It was pointed out to us by Peter Teichner that these surfaces give interesting candidates for ``exotic'' unknotted surfaces. Although each $T_i$ becomes smoothly unknotted after a single stabilization, it is unclear whether each $T_i$ is even topologically unknotted (see the discussion in \autoref{questions}). 
\end{remark}

\fig{120}{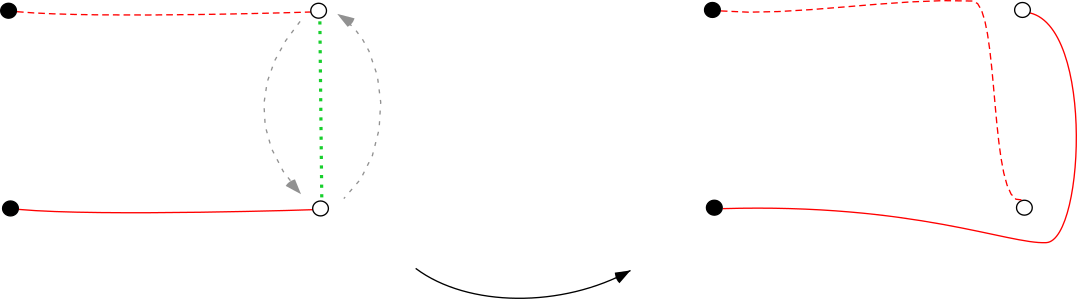}{
    \put(-370,40){\small{$w^*$}}
      \put(-24,125){\small{$p$}}
    \put(-22,47){\small{$p^*$}}
      \put(-149,125){\small{$n$}}
       \put(-315,75){\small{$\gamma$}}
    \put(-148,44){\small{$n^*$}}
     \put(-308,125){\small{$p$}}
    \put(-300,47){\small{$p^*$}}
      \put(-431,124){\small{$n$}}
    \put(-431,44){\small{$n^*$}}
    \put(-105,124){\small{$\tau(w)$}}
    \put(-105,43){\small{$\tau(w^*)$}}
    \put(-370,121){\small{$w$}}
    \put(-250,-12){\small{Braid twist $\tau$}}
    \caption{
        The automorphism $\tau$ of the domain of the immersion
        $f \colon S^2 \looparrowright S^4$ with image $\Sigma$.
    }
    \label{moves}}

\begin{proof}[Proof of \autoref{thm:2}]
    We argue that for $K$ with $\ufw(K)=1$, there is an arc-standard length one regular homotopy from $K$ to the unknot $U$. It then follows from \autoref{std} that $\ust(K)=1$.
    Start by letting $\Sigma$ denote the standard immersion with two oppositely signed double points. Fix a parametrization
    $f \colon S^2 \looparrowright S^4$ of $\Sigma$
    with double point pre-images $\bit{p}=\{p, p^*\}$ and $\bit{n}=\{n, n^*\}$. By definition of the Casson-Whitney number, there is a regular homotopy from $\Sigma$ to $K$ consisting of a single Whitney move along a knotted Whitney disk $V$ with pre-image $f^{-1}(\partial V)$ equal to a pair of ``knotted'' arcs $v, v^* \subset S^2$ with $\partial v = \{p, n\}$ and $\partial v^* = \{p^*, n^*\}$.   
    We claim that there is a standard Whitney disk $W$ one of whose boundary arcs has pre-image equal to $v$. To see this, let $W$ be any standard Whitney disk with $f^{-1}(\partial W)$ equal to the ``standard'' arcs $w, w^* \subset S^2$.
    Consider the map $\tau \colon S^2 \to S^2$
    given by a braid twist about the points $p \cup p^*$, as in \autoref{moves}, that fixes the arc $\gamma$ connecting $p$ to $p^*$ setwise. Since $\Sigma$ is a standard immersion, the loop $f(\gamma)$ bounds an embedded disk (usually referred to in the literature as an \emph{accessory disk}) in $S^4$ away from $\Sigma$.
    It therefore follows from
    \cite[Lemma $3.9$]{schneiderman2019homotopy}\footnote{
    Although Schneiderman and Teichner are working in a different context,
    their Lemma $3.9$ applies in our case since (as they note in the discussion in Section 3.G.) their isotopy is supported locally, in the neighborhood of an accessory disk.}
    that there is an ambient isotopy $\rho \colon S^4 \times I \to S^4$ with $\rho_1(\Sigma)=\Sigma$ carrying the standard Whitney disk $W$ to one whose boundary arcs are the image under $\tau$ of those for $W$, as illustrated in \autoref{moves} as well as \cite[Figure $18$]{schneiderman2019homotopy}. We retain the labels $w$, $w^*$, and $W$ even after such an isotopy occurs. 

    The isotopy $\rho$ can be applied once if necessary, as shown in the top row of \autoref{step3}, so that $\partial w= \{p, n\}=\partial v$ for some choice of labelling of the standard arcs. Note that $w$ and $v$ are now isotopic rel the points in $\bit{p} \cup \bit{n}$ if and only if the loop $w \cup v$ (with either orientation) is null homologous in the annulus $S^2 - \{p^*, n^*\}$. This can be arranged by applying the isotopy $\rho^2$ as shown in the bottom row of \autoref{step3} to insert full twists of $w$ around $p^*$. For the standard Whitney disk $W$ with $w=v$, the regular homotopy from $K$ to $U$ consisting of the finger move that is inverse to the Whitney move along the knotted Whitney disk $V$, followed by the Whitney move along $W$, is arc-standard. 
\end{proof}

\fig{320}{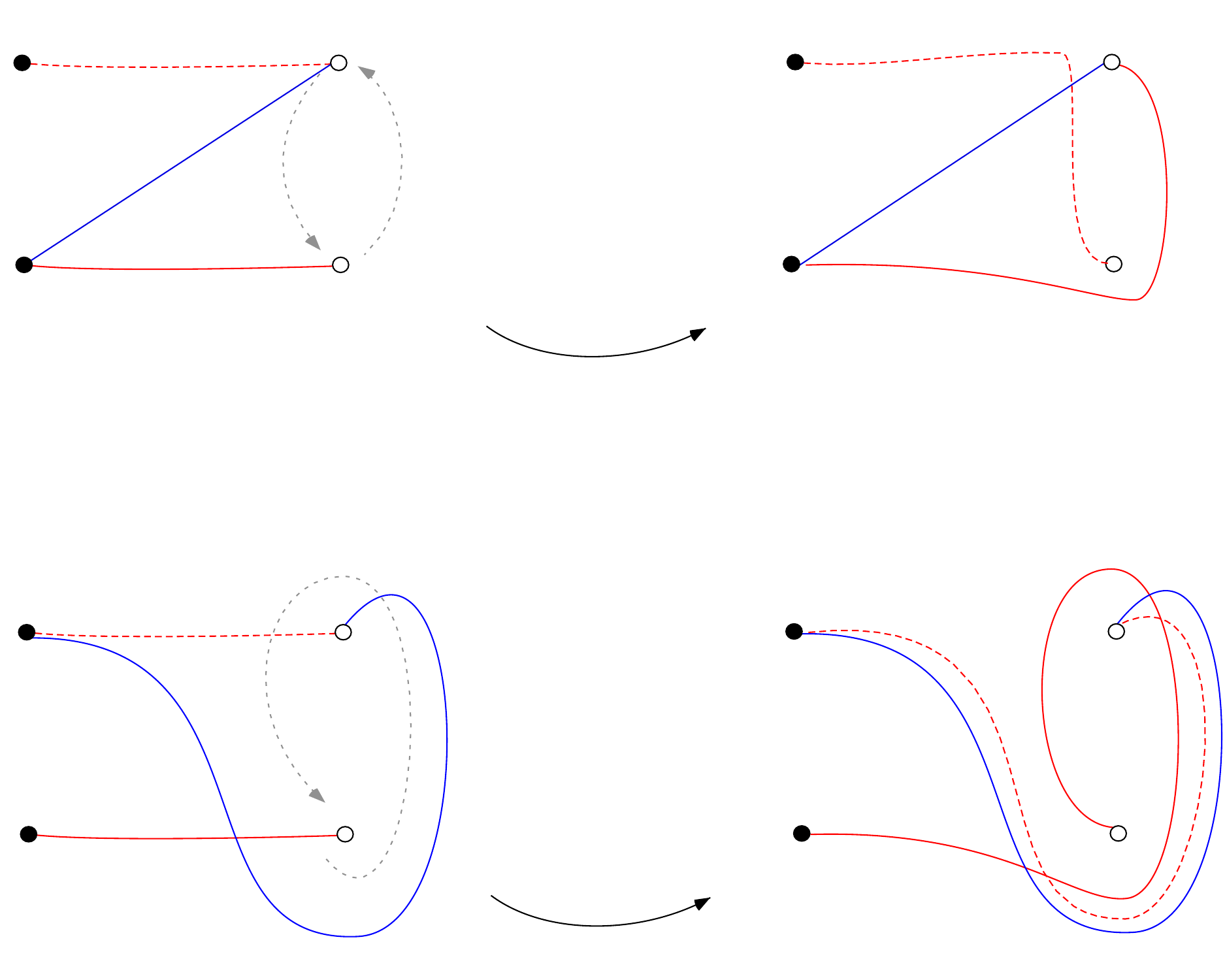}{
    \put(-230,175){\small{(braid twist)}}
    \put(-205,0){\small{$\tau^2$}}
    \put(-207,190){\small{$\tau$}}
      \put(-350,118){\small{$w$}}
         \put(-250,94){\small{$v$}}
         \put(4,94){\small{$v$}}
      \put(-350,52){\small{$w^*$}}
        \put(-350,304){\small{$w$}}
        \put(-95,270){\small{$v$}}
        \put(-350,270){\small{$v$}}
      \put(-350,238){\small{$w^*$}}
     \put(-395,121){\small{$n$}}
      \put(-395,55){\small{$n^*$}}
      \put(-395,307){\small{$n$}}
      \put(-395,241){\small{$n^*$}}
        \put(-292,121){\small{$p$}}
      \put(-290,55){\small{$p^*$}}
      \put(-292,307){\small{$p$}}
      \put(-290,241){\small{$p^*$}}
           \put(-144,121){\small{$n$}}
      \put(-144,55){\small{$n^*$}}
      \put(-144,307){\small{$n$}}
      \put(-144,241){\small{$n^*$}}
        \put(-40,122){\small{$p$}}
      \put(-40,55){\small{$p^*$}}
      \put(-40,307){\small{$p$}}
      \put(-40,241){\small{$p^*$}}
       \put(-106,116){\small{$\tau^2(w)$}}
      \put(-111,52){\small{$\tau^2(w^*)$}}
      \put(-100,308){\small{$\tau(w)$}}
      \put(-100,238){\small{$\tau(w^*)$}}
    \caption{The ambient isotopies $\rho$ and $\rho^2$
    from the proof of \autoref{thm:2} used to move the standard disk $W$
    to one with $w=v$.
    The result of each isotopy is illustrated from the perspective
    of the induced maps $\tau$ and $\tau^2$ on the domain of the immersion
    $f \colon S^2 \looparrowright S^4$ with image $\Sigma$.}
    \label{step3}
}

\section{Geometric upper bounds}
\label{geo}

The Casson-Whitney unknotting number can be bounded from above geometrically, by constructing simple regular homotopies to the unknot. We do this for some well-known families of spheres. 

\begin{definition} \label{ribbondef} 
    A \emph{ribbon 2-knot}
    is formed from $n$ stabilizations of the $(n+1)$-component
    unlink $U_1 \sqcup \dots \sqcup U_{n+1}$ in $S^4$, as in 
    \autoref{fig:ribbon_2_knot}.
    The minimal number $n$ needed to put a ribbon $2$-knot $K$
    in this form is called the \emph{fusion number} of $K$,
denoted $\fus(K)$.  
\end{definition} 

\begin{figure}
    \centering
    \includegraphics[width=0.6\linewidth]{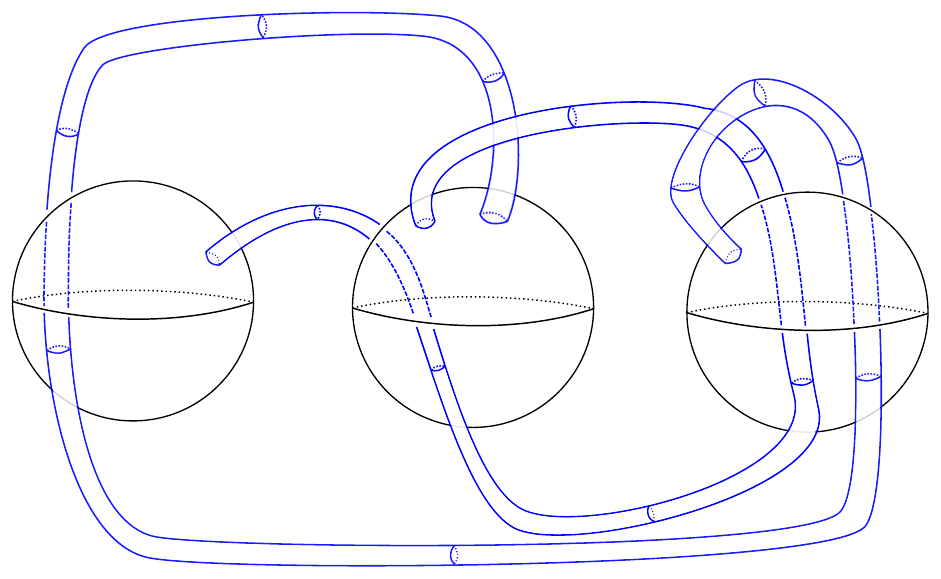}
        \put(-230,60){\small{$U_1$}}
        \put(-130,60){\small{$U_2$}}
        \put(-60,60){\small{$U_3$}}
    \caption{Ribbon 2-knot}
    \label{fig:ribbon_2_knot}
\end{figure}

\begin{remark}[Tube map]
    Satoh proved in \cite{satoh2000tube} that every ribbon 2-knot
    is the tube of a virtual arc.
    Essentially, one can use virtual diagrams
    to make a shorthand picture for a broken surface diagram of a ribbon 2-knot.
    In this language, changing a virtual crossing to a positive
    or a negative classical crossing is achieved by a finger move
    and then a Whitney move on its tube (the analog in this setting of the homotopy in \autoref{fig:tube_map_crossing_change}).
    Thus if $K$ is a ribbon 2-knot and $k$
    is any virtual arc such that $\operatorname{Tube}(k)=K$,
    any sequence of crossing changes which unknots $k$
    as a virtual (or welded) arc yields a sequence of
    finger and Whitney moves which unknots $K$. 
\end{remark}

\fig{100}{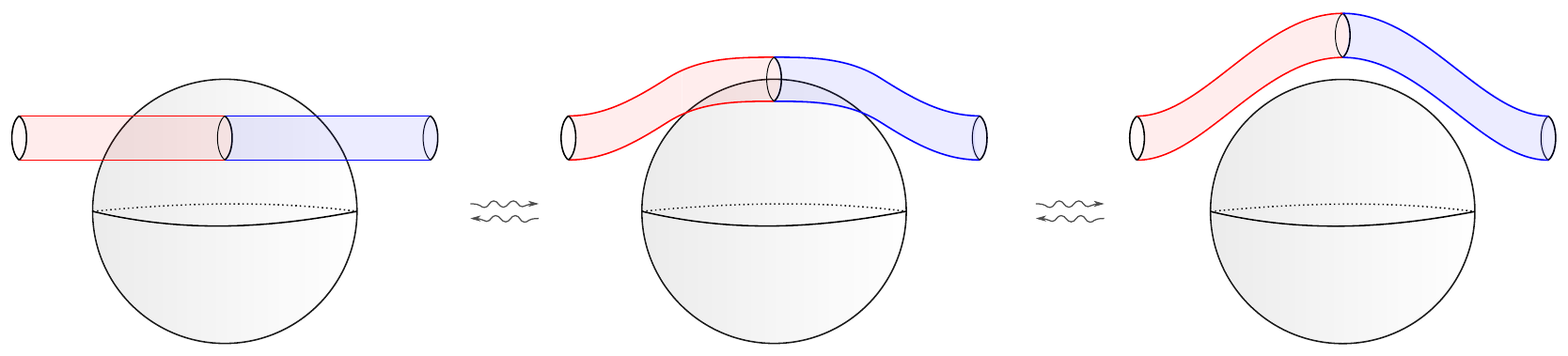}{
    \put(-150,47){\small{WM}}
    \put(-147,27){\small{FM}}
    \put(-308,47){\small{FM}}
    \put(-310,27){\small{WM}}
    \put(-383,22){\small{$U_i$}}
    \put(-228,22){\small{$U_i$}}
    \put(-68,22){\small{$U_i$}}
    \caption{A regular homotopy of a ribbon $2$-knot $K$, as in \autoref{ribbondef}, supported near one component $U_i$ of the unlink and one guiding arc of a stabilization. The various shadings of $K$ suggest its fourth coordinates -- so, the red and blue portions of the surface are disjoint from the black ones.  The homotopy consists of one finger move followed by one Whitney move, and (thought of from left to right) has the effect of removing a meridian of $U_i$ from the word in $\pi (U_1 \sqcup \dots \sqcup U_{n+1})$ giving the homotopy class of the guiding arc of the stabilization.}
    \label{fig:tube_map_crossing_change}}

In \cite{miyazaki1986relationship}, Miyazaki proved that 
$\ust(K) \le \fus(K)$ for a ribbon $2$-knot $K$. We prove the corresponding statement for the Casson-Whitney number in \autoref{thm:fusion_upper_bound} below. The proof is inspired by Miyazaki's, however the argument is subtler, so we first sketch Miyazaki's argument. 

Let $K$ be a ribbon 2-knot, formed by stabilizing the unlink $U_1 \sqcup \dots \sqcup U_{n+1}$ along guiding arcs connecting consecutive components $U_i$ and $U_{i+1}$, as in \autoref{ribbondef}.
The `obvious' stabilizations of $K$ which fuse $U_i$ to $U_{i+1}$ as in \autoref{fig:ribbon_2_knot_with_trivial_bands}
result in an unknotted surface $K^\prime$: thinking of $K^\prime$ as first formed by attaching this second set of tubes, and then the original tubes defining $K$, produces the same surface $K^\prime$. However, this is clearly unknotted, since the `obvious' tubes result in an unknotted sphere, and so the original tubes are (trivial) stabilizations of an unknotted surface, which must be unknotted.

Although one could perform $\fus(K)$ finger moves to abelianize the group of $K$, the rest of Miyazaki's argument breaks down in our case: if one thinks of the finger moves as performed on the $n+1$ component unlink, then the group of the complement is not abelian since it takes $\binom{n+1}{2}$ finger moves to abelianize the group of the complement of this unlink. To remedy this, we think of all but one of the tubes as already attached and proceed by induction, allowing us to work with 2 components instead of $n+1$.

\begin{figure}
    \centering
    \includegraphics[width=0.6\linewidth]{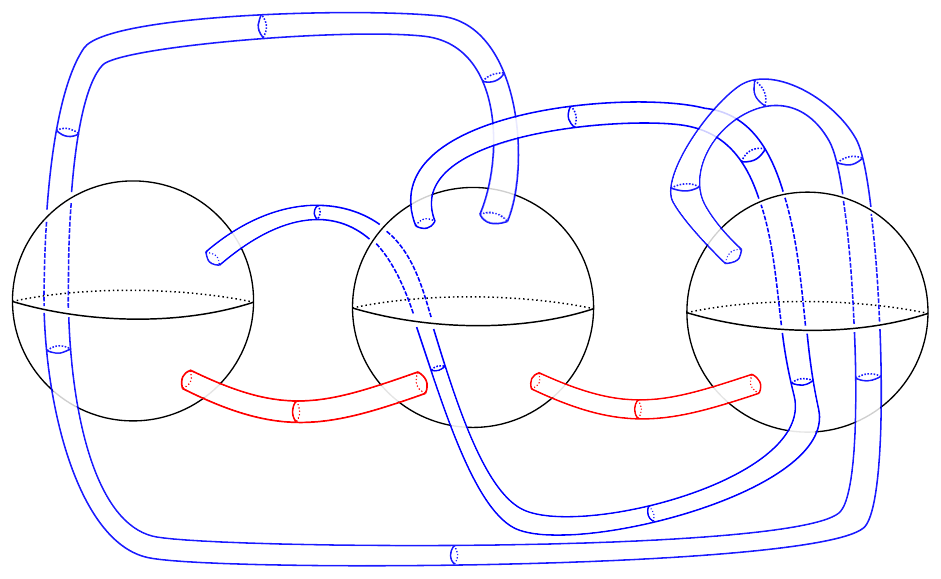}
        \put(-240,62){\small{$U_1$}}
        \put(-134,62){\small{$U_2$}}
        \put(-60,62){\small{$U_3$}}
    \caption{Miyazaki's proof that $\ust(K) \le \fus(K)$:
    The red stabilization of the black unlink is the unknot $U$. 
   Hence, the guiding arcs for the blue stabilizations are isotopic rel boundary to trivial arcs in the complement of $U$.}
    \label{fig:ribbon_2_knot_with_trivial_bands}
\end{figure}

\fusionbound*

\begin{proof}
    Let $n = \fus(K)$.  Then, as in \autoref{ribbondef}, the knot $K$ can be obtained from the unlink $\mathcal{U}= U_1 \sqcup \dots \sqcup U_{n+1}$ by $n$ stabilizations along guiding arcs $\alpha_1, \dots, \alpha_n$. 
    After an isotopy, we may assume that each $\alpha_i$ connects $U_i$ to $U_{i+1}$
    as in \autoref{fig:ribbon_2_knot}.
    Let $L$ be the  $2$-component link obtained by stabilizing the unlink only along the guiding arcs
    $\alpha_2, \dots ,\alpha_n$.
    Recall from \autoref{sec:background}, in particular \autoref{lem:pi_1_handle}, that 
    \begin{align*}
        \pi L \cong
        \langle m_1,m_2,\dots,m_{n+1} 
        \mid \;
         m_j^{g_j} = m_{j+1} \text{ for } 1<j<n+1 \rangle
    \end{align*}
    where the $g_j$ correspond to the guiding arcs $\alpha_i$ as in \autoref{defguidingarc}, and
    $m_1, \dots, m_{n+1} \in \pi \mathcal U$ are meridians of each component $U_1, \dots, U_{n+1}$. 
    
    Perform $n$ finger moves to $L$ along trivial guiding arcs 
    from $U_{n+1}$ to $U_{i}, i \le n$
    as in \autoref{fig:ribbon_proof_frame_2}
    and call the resulting immersed $2$-component link $S$.
    By \autoref{lem:pi_1_finger_move}, we have made $m_{n+1}$
    commute with $m_i$ for all $i<n+1$,
    therefore by considering the previous relations we see that
    $\pi S \cong \mathbb{Z} \oplus \mathbb{Z}$ generated by $m_1$ and $m_2$.  
    
    Now consider the element $g_1$ corresponding
    to the guiding arc $\alpha_1$ as in
    \autoref{fig:ribbon_proof_frame_2}.
    Since $\pi S$ is both abelian and generated by
    $m_1$ and $m_2$, by \autoref{rem:boundarytwist},
    $\alpha_1$ is isotopic to a trivial arc between $U_1$ and $U_2$
    as shown in \autoref{fig:ribbon_proof_frame_3}.
    We can now undo the finger move
    that intersects $U_1$ (i.e.\ do the Whitney move) 
    and proceed, by the same reasoning,
    to straighten out all of the other arcs $\alpha_i$
    to be trivial arcs as in
    \autoref{fig:ribbon_proof_frame_4}.
    Proceeding in this way unknots $K$
    (by trivializing the guiding arcs $\alpha_1,  \dots, \alpha_n$)
    with $n$ finger moves and $n$ Whitney moves.  
\end{proof}

\begin{figure}[ht]
    \begin{subfigure}{0.475\textwidth}
      \centering
      \begin{overpic}[width=0.95\textwidth]{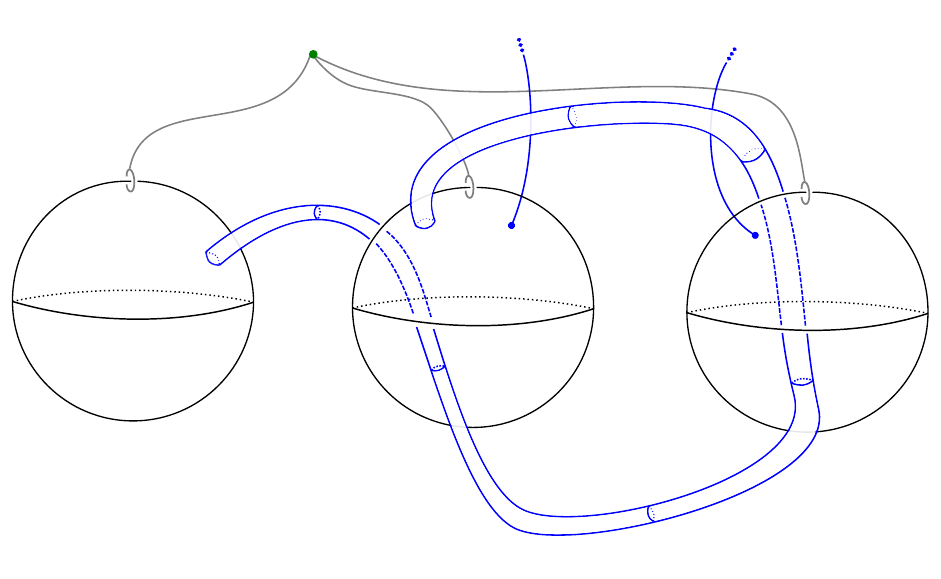}
      \end{overpic}
      
      \caption{Ribbon presentation of
      $K$ with the first stabilization
      along the guiding arc $\alpha_{1}$
      drawn as a blue tube.
      Only
      the guiding arcs (blue) of all the other
      stabilizations are shown.
      \label{fig:ribbon_proof_frame_1}}
    \end{subfigure}
    \hfill
    \begin{subfigure}{0.475\textwidth}
      \centering
      \begin{overpic}[width=0.95\textwidth]{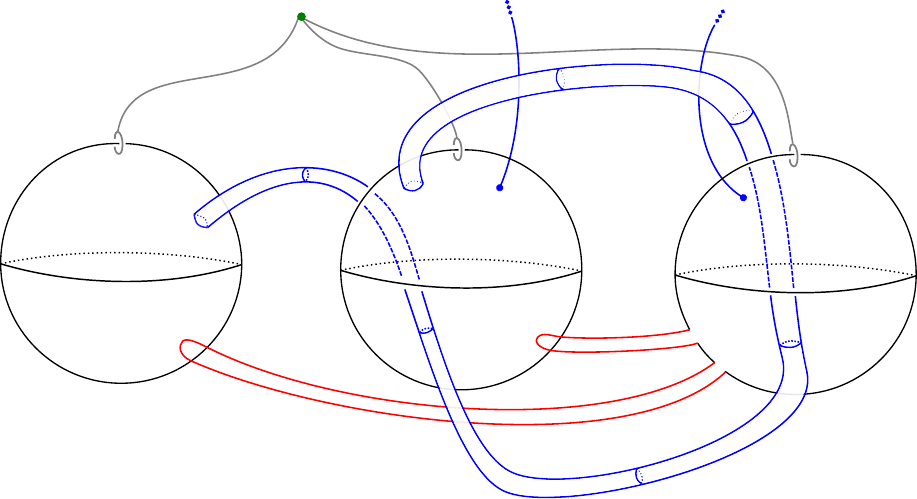}
      \end{overpic}
      
      \caption{Finger moves (red) along trivial arcs.
      \label{fig:ribbon_proof_frame_2}}
    \end{subfigure}
    \vskip\baselineskip
    \begin{subfigure}{0.475\textwidth}
      \centering
      \begin{overpic}[width=0.95\textwidth]{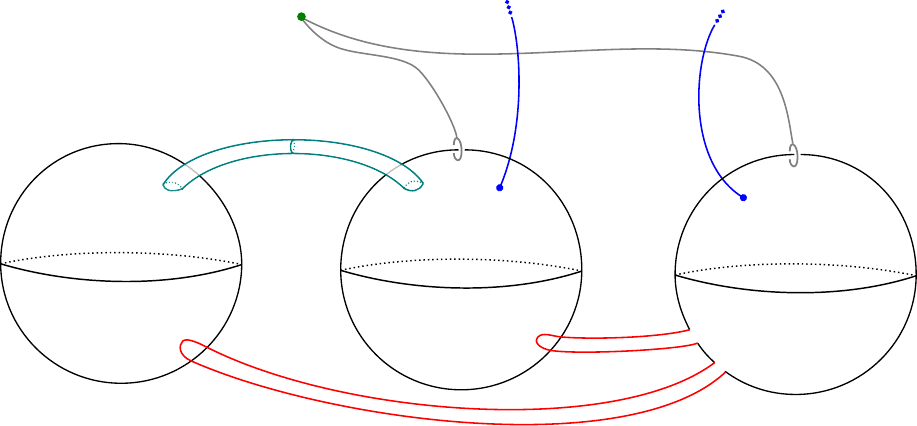}
      \end{overpic}
      
      \caption{The guiding arc of the first
      stabilization
      is isotopic to a trivial arc in the
      complement of the immersion.
      \label{fig:ribbon_proof_frame_3}}
    \end{subfigure}
    \hfill
    \begin{subfigure}{0.475\textwidth}
      \centering
      \begin{overpic}[width=0.95\textwidth]{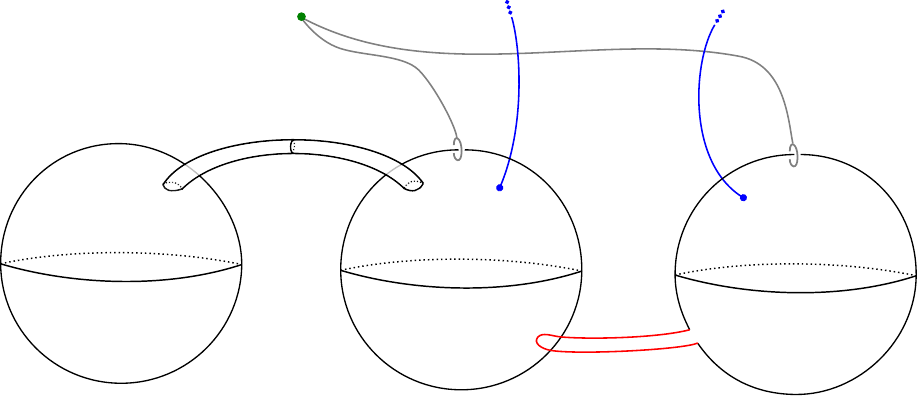}
      \end{overpic}
      
      \caption{Undo the first finger move
      in order to proceed inductively and trivialize the
      remaining stabilizations.
      \label{fig:ribbon_proof_frame_4}}
    \end{subfigure}
    
\caption{
    Illustrating the proof of
    the fusion number upper bound in
    \autoref{thm:fusion_upper_bound}.
    }
    \label{fig:ribbon_proof}
\end{figure}

\begin{remark}
    \label{rem:doublystd}
    The regular homotopy in the proof of \autoref{thm:fusion_upper_bound} above
    has standard and knotted Whitney disks whose entire boundaries agree
    (so, they are ``doubly'' arc-standard).
    However, since the homotopy of the tubes in the proof
    interacts with the standard Whitney disks,
    the interiors of the standard and knotted Whitney disks are necessarily distinct.
\end{remark}

The inequality $ \ufw(K) \le \fus(K)$ from \autoref{thm:fusion_upper_bound} is sometimes strict. For instance, let $k$ be a 1-knot with unknotting number $u(k) = 1$ and
meridional rank $\mu(k) > 2$ (in fact, it is shown in \cite{baader2019symmetric} that there exist unknotting number one knots with arbitrarily large meridional rank).
Then, we will see that the spun knot $K= \tau^0(k)$, defined below, has $\ufw(K) = 1$
by \autoref{thm:lower_bound_classical_unknotting}.
Moreover, the isomorphism between $\pi K$ and $\pi k$
preserves the meridians, and so the meridional ranks of $K$ and $k$ are equal.
This gives $ \ufw(K) < \fus(K)$, since spun knots are ribbon and the meridional rank of any ribbon 2-knot is less than or equal to one more than its fusion number.

A generalization of this family of spheres for which $\ufw$ is particularly convenient to analyze is constructed by `spinning' $3$-balls containing properly embedded knotted arcs through an open book decomposition of $S^4$. 

\begin{definition}[Twist spun knots]
    \label{spindef}
    Given a $1$-knot $k \in S^3$, let $k'$ be the properly embedded knotted arc in the $3$-ball 
    whose tubular neighborhood $N(k')$ has complement
    $B^3 - N(k')$ diffeomorphic to $S^3 - N(k)$.
    For $n \in \mathbb{Z}$, consider the quotient
    \begin{align*}
     \quad ^{\textstyle (B^3, k') \times S^1}\Big/_{\textstyle (r_{n,\theta}(x), \theta) \sim (x, 0), ~ x \in \partial B^3, ~ \theta \in [0, 2\pi]}
    \end{align*}
    \noindent where $r_{n,\theta}: B^3 \to B^3$ denotes the ambient isotopy rotating $B^3$ 
    by an angle of $n \theta$
    about an axis with endpoints $\partial k' \subset \partial B^3$. For each $n$, this quotient space is diffeomorphic to $S^4$, and gives an open book decomposition with binding an unknotted $2$-sphere $U$ and 3-ball pages $B^3_\theta$ for all $\theta \in S^1$. The quotient of $k' \times S^1$ is a 2-sphere $\tau^n k \subset S^4$ called the \bit{$n$-twist spin} of $k$. 
\end{definition}

Due to Artin \cite{artin}, the collection of $0$-twist spun knots, often simply called `spun knots', were the first examples of non-trivially knotted spheres in $S^4$. Artin proved that the group of the spun knot $\tau^0(k)$ is isomorphic to the group of the classical knot $k$, showing that every 1-knot group is also a 2-knot group. 

Twist spinning was introduced by Zeeman in \cite{zeeman1965twisting} as a generalization of the spinning construction. For $n \not =0$, Zeeman proved that the resulting twist spun knot is fibered by the $n$-fold cyclic branched cover of $k$. Thus $\tau^{\pm1} k$ is unknotted, for all $k$. Twist spun knots provide a large generalization of spun knots. Cochran proved that any non-trivial twist spun knot $\tau^n k$ with $n\neq0$ is not ribbon \cite{cochran}, in contrast to spun knots, which are always ribbon.

\fig{160}{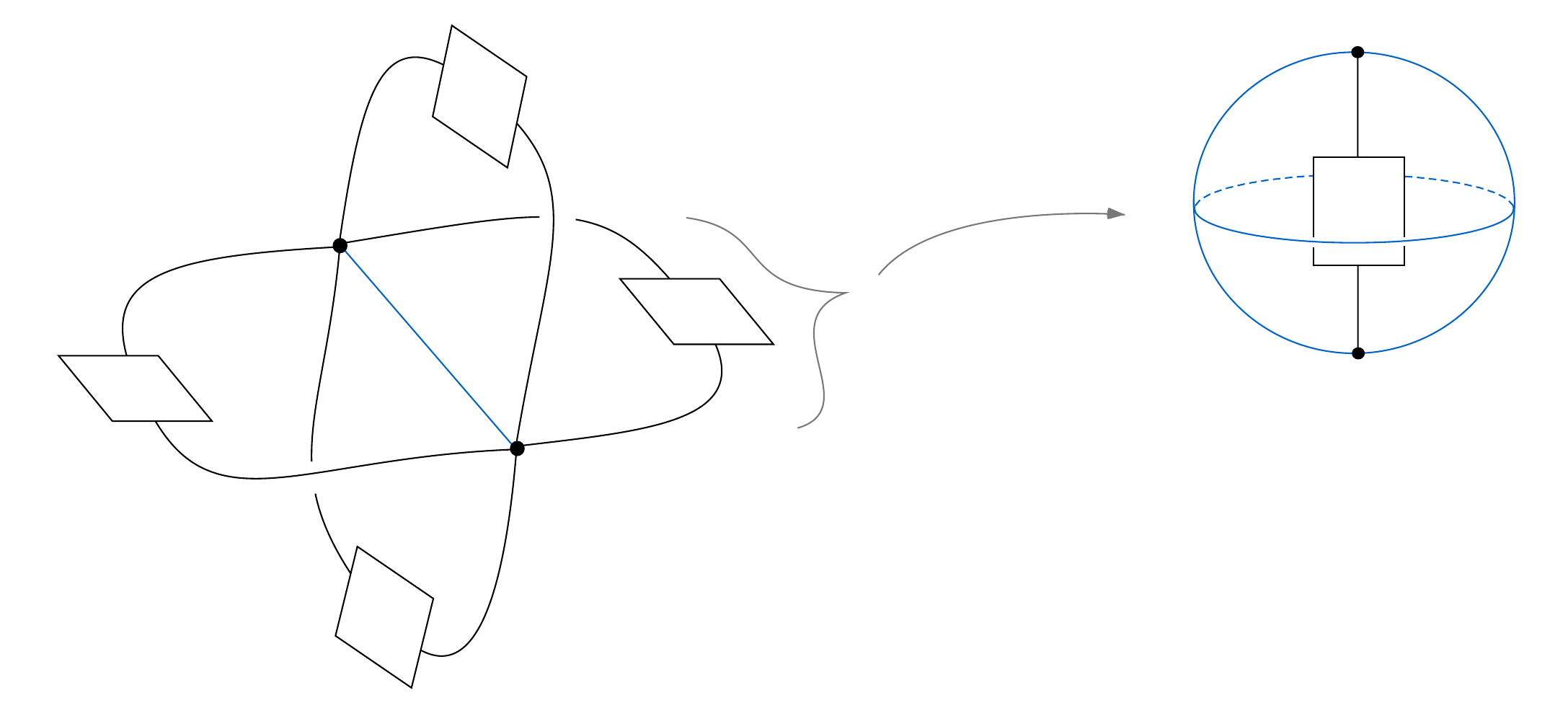}{
  \put(-336,68){\small{$-k'$}}
  \put(-53,65){\small{$B^3_\theta$}}
  \put(-255,83){\small{$U$}}
  \put(-84,80){\small{$U$}}
  \put(-200,85){\small{$k'$}}
  \put(-50,110){\small{$k'$}}
  \put(-252,135){\small{$k'$}}
  \put(-278,16){\small{$-k'$}}
    \caption{A schematic for the spin $\tau^0 k$
    and the twist spins $\tau^{n} k$ of a $1$-knot $k$. The open book decomposition of $S^4$ from Definition \autoref{spindef} has been knocked down one dimension on the left, and so the blue $2$-sphere $U$ around which the knotted arc is spun is instead a blue circle (only an arc of which is drawn). 
    Note that viewing a page $B^3_\theta$ from the
    ``opposite side'' reverses its orientation,
    and therefore also the orientation of $k$.}
    \label{spinpic}}

\begin{lemma}
    \label{twistregionfm} 
    Fix two parallel strands of a $1$-knot $k \subset S^3$, and let $k_s$ denote the knot obtained by inserting $s$ full twists into these strands. Then, for any $n \in \mathbb{Z}$, there is a length one regular homotopy between the twist spins $\tau^n(k)$ and $\tau^n(k_s)$. 
\end{lemma}

\begin{proof}
    By performing a finger move on $\tau^{n}(k_{s+1})$
    along the arc $\alpha_{s+1} \subset B^{3}_{0}$
    as in \autoref{fmspunpic}, we obtain an immersed surface
    $\Sigma_{s+1}$ that is also obtained by a finger move
    to $\tau^{n}(k_{s})$ along $\alpha'_{s} \subset B^3_{\pi}$,
    which we will also denote $\Sigma'_{s}$,
    so that $\Sigma'_{s} = \Sigma_{s+1}$. Note that the twist parameter $n$ is unchanged since the twisting can be assumed to occur in a small interval in $S^1$ away from the double points of the immersion, the knotted arc $k_s$ twists $n$ times in both $\Sigma'_{s}$ and $\Sigma_{s+1}$. 
    
    By instead performing a finger move on $\tau^{n}(k_{s})$ along the arc
    $\alpha_{s} \subset B_{0}^3$,
    we obtain a surface $\Sigma_{s}$ where $\Sigma_{s} = \Sigma'_{s}$
    by rotation.
    Similarly by performing a finger move
    to $\tau^{n}(k_{s-1})$ along the arc $\alpha'_{s-1} \subset B'_{\pi}$
    we obtain a surface $\Sigma'_{s-1}$ with $\Sigma'_{s-1} = \Sigma_{s}$.
    
    Thus, we have equivalent immersed surfaces
    \[
        \ldots = \Sigma_{s+1} = \Sigma'_{s} = \Sigma_{s} = \Sigma'_{s-1} = \ldots
    \]
    so that for any $s, t \in \ZZ$, the two knots
    $\tau^{n}(k_{s})$, $\tau^{n}(k_{t})$ are related by a single finger and
    Whitney move.
\end{proof} 



Some implications of \autoref{twistregionfm} are immediate.
For instance, as there is a length one regular homotopy between $1$-knots
related by a single crossing change, we obtain the following corollary. 

\begin{corollary}
    \label{thm:lower_bound_classical_unknotting}
    Let $k \colon S^1 \hookrightarrow S^3$
    be a classical knot.
    For any twist spin $\tau^n k$,
    $\ufw(\tau^n k)\leq u(k)$,
    where $u(k)$ is the classical unknotting number of $k$. 
\end{corollary}

\fig{260}{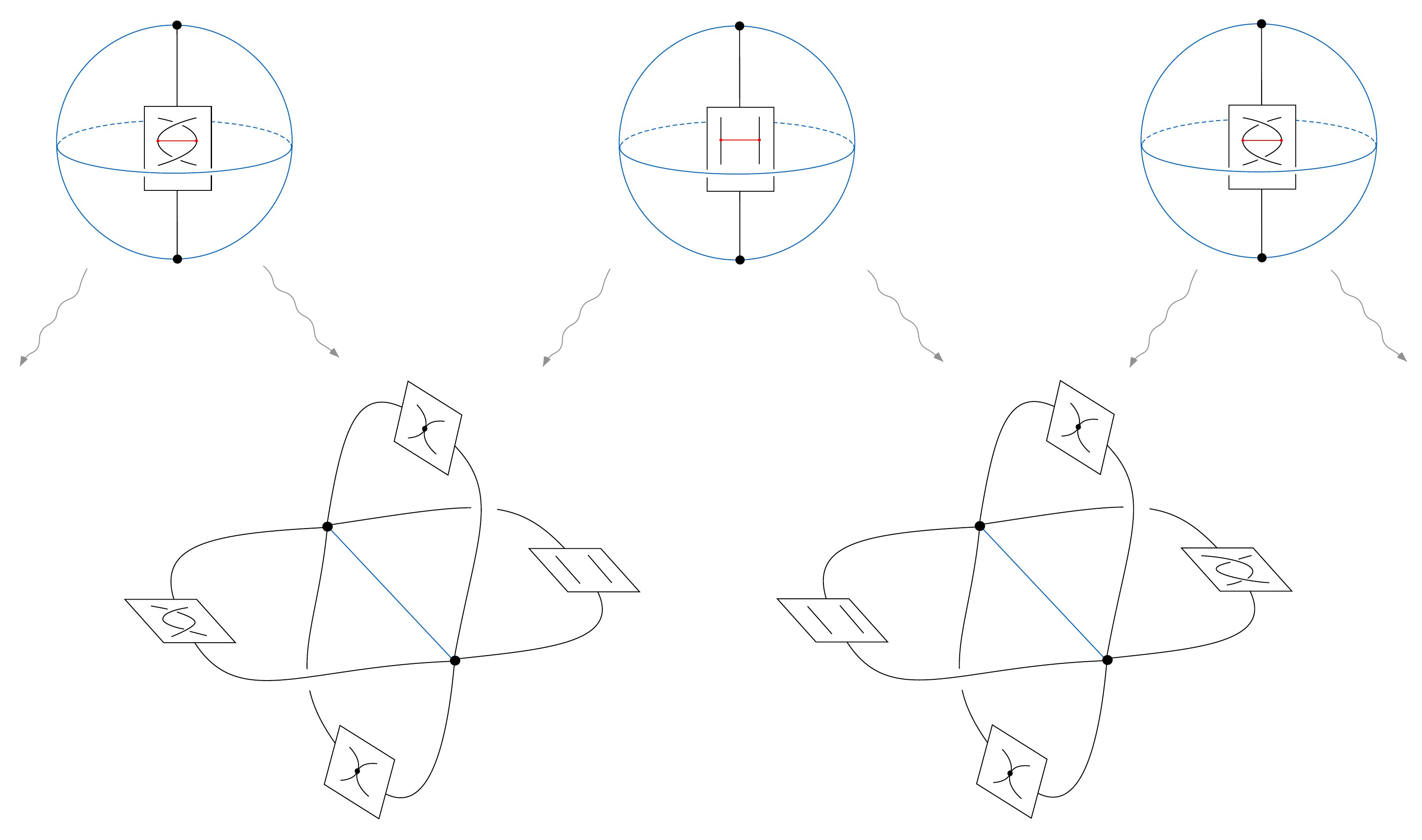}{
    \put(-445,160){\small{FM}}
    \put(-410,160){\small{$\alpha_{s+1} \subset B^3_{0}$}}
    \put(-337,160){\small{FM}}
    \put(-285,160){\small{FM}}
    \put(-252,160){\small{$\alpha'_s \subset B^3_{\pi}$; $\alpha_s \subset B^3_{0}$}}
    \put(-150,160){\small{FM}}
    \put(-101,160){\small{FM}}
    \put(-73,160){\small{$\alpha'_{s-1} \subset B^3_{\pi}$}}
    \put(-10,160){\small{FM}}
    \put(-340,-6){\small{$\Sigma_{s+1}~=~\Sigma'_s$}}
    \put(-135,-6){\small{$\Sigma_s~=~\Sigma'_{s-1}$}}
    \put(-225,60){$=$}
    \put(-445,60){$\cdots \; =$}
    \put(-25,60){$= \; \cdots$}
     \put(-456,140){$\cdots$}
    \put(0,140){$\cdots$}
    \caption{
        Top row: The intersections of the twist spins
        $\tau^n(k_{s+1})$, $\tau^n(k_s)$ and $\tau^n(k_{s-1})$
        with $B^3_0$ and $B^3_\pi$, as in \autoref{spindef}. Only the relevant crossing of each cross section is drawn.
        Bottom row: Schematics of the immersed spheres
        (compare to the embedded spheres in \autoref{spinpic})
        obtained by doing finger moves along the red
        guiding arcs in each diagram of the top row. Again, only the relevant crossing of the cross section in each page $B_\theta^3$ is shown. The isotopy between the two immersions $\Sigma'_s$
        and $\Sigma_s$ is via a rotation by $\pi$.
    }
    \label{fmspunpic}}

Although we are not aware of another instance of
\autoref{thm:lower_bound_classical_unknotting} in the literature,
the analogous result for $\ust$ was proved by Satoh \cite{satoh2004unknotting},
and also follows from Proposition $9$ of  \cite{baykur2016knottedsurfaces}.
Furthermore, Satoh proved in \cite{satoh2004unknotting}
that for any twist spin of a $b$-bridge knot,
the stabilization number is strictly less than $b$. 
When $b=2$, we prove that the same inequality holds for the
Casson-Whitney number of any twist spin. 

\begin{theorem}
    \label{thm:2bridgecw}
   If the twist spin $\tau^n k$ of a $2$-bridge knot $k$ is not unknotted, then it has $\ufw(\tau^n k)=1$. 
\end{theorem}

\begin{proof}
    Since $k$ is $2$-bridge, it can be put into normal form \cite{conway1970enumeration}
    with non-zero twist parameters $a_1, b_1, \dots, a_m, b_m$ indicating the number of half twists in each region, as in \autoref{2bridgepic}. In fact, we may assume that the terms $a_i$ and $b_i$ are all even\footnote{It was pointed out in \cite{baaderetal2019} that this can be shown using the continued fraction notation for $2$-bridge knots, for instance see \cite{kawauchi1996}.}. Start by performing a finger move of $\tau^n k$ along the red guiding arc pictured in the leftmost diagram of \autoref{2bridgepic}, at some angle $\theta \in S^1$. This results in an immersed sphere $\Sigma_k$ that we will prove is the standard immersed sphere gotten by one finger move on the unknot, by induction on the number $m$ of twist region pairs $a_i, b_i$. 

\fig{90}{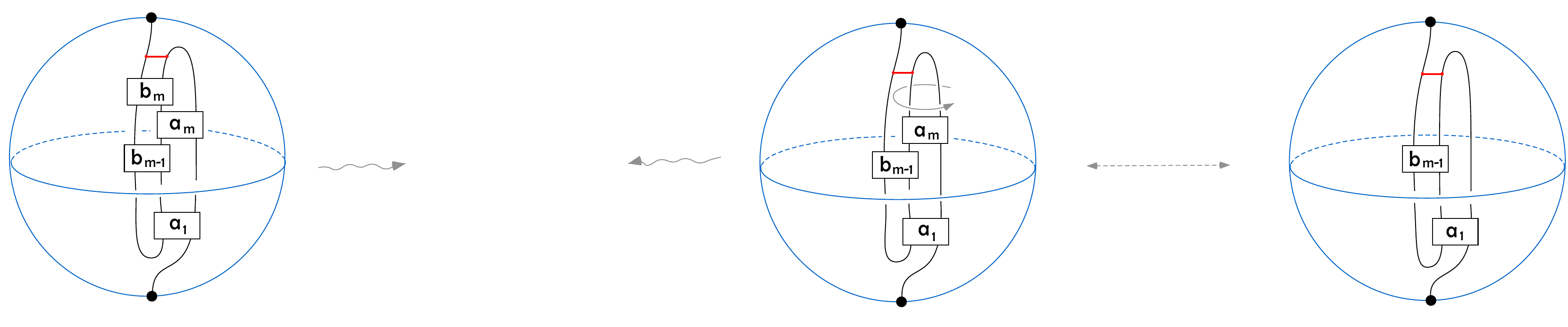}{
  \put(-130,50){\small{Isotopy}}
  \put(-340,50){\small{FM}}
  \put(-255,50){\small{FM}}
    \put(-311,41){$\Sigma_k ~=~ \Sigma$}
    \caption{The intersection of the twist spins $\tau^n k$ (left)
    and $\tau^n \widehat k$ (middle and right) with the $3$-ball $B^3_\theta$ from \autoref{spindef},
    for each $\theta \in S^1$.}
    \label{2bridgepic}}

    When $m=0$, the knot $k$ and hence also its $n$-twist spin $\tau^n k$ are unknotted.  Therefore, $\Sigma_k$ is the standard immersion by definition. So, suppose $m\geq1$. Let $\widehat k$ denote the $2$-bridge knot with two fewer twist parameters
    $a_1, b_1, \dots, a_{m-1}, b_{m-1}$ and assume as the inductive hypothesis that the immersed sphere gotten by a finger move of the twist spin $\tau^n \widehat k$ along the red guiding arc pictured in the rightmost diagram of \autoref{2bridgepic} is ambiently isotopic to the standard immersion $\Sigma$. Observe that the guiding arc for this finger move is isotopic to the red guiding arc shown in the middle diagram of \autoref{2bridgepic}; therefore doing a finger move of $\tau^n \widehat k$ along this arc also gives the standard immersion $\Sigma$. 
    
    Now, since the knots $k$ and $\widehat k$ differ only along a single twist region, by \autoref{twistregionfm}, the twist spins $\tau^n k$ and $\tau^n \widehat k$ must give ambiently isotopic immersions after one finger move. Indeed, the guiding arcs for the finger moves of $\tau^n k$ and $\tau^n \widehat k$ that are used in the proof of \autoref{twistregionfm}
    (i.e.\ those from \autoref{fmspunpic})
    are equal to the red guiding arcs from the middle and leftmost diagrams
    of \autoref{2bridgepic}.
    It follows that $\Sigma_k$ is ambiently isotopic
    to the standard immersion $\Sigma$, as desired.
\end{proof}

\begin{remark}
    The quantities in the inequality \autoref{thm:lower_bound_classical_unknotting} can be arbitrarily far apart. For instance, each nontrivial $(2,p)$ torus knot $k_p$ has unknotting number $|p-1|/2$ by \cite{kronheimer_mrowka_1993} and bridge number equal to two. Therefore, by \autoref{thm:2bridgecw}, $\ufw(\tau^n k_p)= 1 < u(k_p)$ for any nontrivial twist spin of $k_p$, whenever $p\geq 5$.  
\end{remark}

\begin{remark}
    Satoh's proof \cite{satoh2004unknotting} that $\ust(K) \leq b-1$ for any twist spin of a $b$-bridge knot relies on the fact that $1$-handles can be slid over one another when $b>2$. However, this cannot be done with finger moves, making our proof of \autoref{thm:2bridgecw} difficult to extend to knots with higher bridge number.
\end{remark}

\section{Algebraic lower bounds}
\label{sec:algebraic_lower_bounds}

In this section, we discuss the {\it algebraic Casson-Whitney number} $\afw(K)$ of a $2$-knot $K$, the minimal number of meridian-commuting relations which abelianize the knot group of $K$ (see \autoref{def:algcw} for the precise definition). This algebraic invariant is the sharpest lower bound we are aware of for the Casson-Whitney number $\ufw$, and in \autoref{sec:classical_unknotting} we show that it is also a lower bound for the classical unknotting number. It is clear that $\asta(K)\leq\afw(K)$, as stabilization relations identify two meridians, while finger move relations merely force them to commute (see \autoref{subsec:groupcalcs} for a thorough description of the effects of the corresponding geometric operations on the knot group).
This subtle difference is used to prove \autoref{thm:3},
in which we give $2$-knots for which $\asta(K)<\afw(K)$
and for which this difference is realized geometrically. 

\subsection{Previously known results}
\label{sec:previousalg}

The minimal number of generators of the Alexander module, called the \bit{Nakanishi index $m(K)$}, is a classical lower bound for the unknotting number of 1-knots \cite{nakanishi1981unknotting}. In \cite{miyazaki1986relationship}, \cite{miller2019stabilization} it is shown that the Nakanishi index is also a lower bound for the stabilization number $\ust(K)$ of 2-knots. Indeed, any set of relators which abelianize the group of a knot must normally generate its commutator subgroup.
It is an exercise (cf.\ \cite[Exercise 7.D.5]{rolfsen}) to show
that the images of these relators generate the Alexander module.

A subtler but sharper bound for the classical unknotting number
is the \bit{Ma-Qiu index $a(K)$},
defined as the minimal number of relations
needed to abelianize the knot group \cite{maqiu}. In fact, the algebraic stabilization number $\asta(K)$ (\autoref{def:algstab}), the minimal number of stabilization relations needed to abelianize the knot group, is also a lower bound for the classical unknotting number. This is evident from the proof of \cite{maqiu}, and also by combining the inequalities from \autoref{basic} and \autoref{thm:lower_bound_classical_unknotting}, applied to the spin of a 1-knot $k$: $\asta(k)=\asta(\tau^0 k)\leq \afw(\tau^0 k)\leq \ufw(\tau^0 k)\leq u(k)$.

As noted in \autoref{basic}, the algebraic stabilization number, $\asta(K)$, is a natural lower bound for the stabilization number $\ust(K)$. This is also studied in \cite{kanenobu1996weak}, where it is called the {\it weak unknotting number}. In this section, we investigate the {\it algebraic Casson-Whitney number}, which is a natural lower bound for the Casson-Whitney number. By \autoref{thm:lower_bound_classical_unknotting}, it also provides a lower bound for the classical unknotting number via spinning, which as we show in \autoref{thm:algadd} is sharper than the bounds provided by $\asta(K)$ and $\ust(K)$. 

We summarize the previously known results regarding these invariants
in the proposition below. 

\begin{proposition}[Kanenobu, Ma-Qiu, Miyazaki, Nakanishi]
    \label{previousinequalities}
    If $k$ is a 1-knot, then
    \[
        m(k) \leq a(k) \leq \asta(k) \leq u(k).
    \]    
    If $K$ is a 2-knot, then
    \[
        m(K)\leq a(K)\leq \asta(K) \leq \ust(K).
    \]
\end{proposition}

As pointed out in \cite{maqiu}, the first inequality above is often strict: the Ma-Qiu index is positive whenever $\pi K$ is not abelian, but the Alexander module and hence the Nakanishi index can be zero for nontrivial knots, e.g.\ Alexander polynomial one 1-knots. While $m(k),a(k)$, and $\asta(k)$ are known to be nonadditive on certain classical knots (see the end of \autoref{sec:algmain}), we are unaware of any classical knots for which $\afw$ is nonadditive. We show in \autoref{sec:nonadd} that is nonadditive on certain 2-knots. 

\subsection{The algebraic Casson-Whitney number}
\label{sec:algmain}

Recall from \autoref{subsec:groupcalcs} that each finger move on a 2-knot $K$ adds a relation of the form $[x,y]=1$, where $x, y$ are meridians of $K$. As noted after \autoref{meridian}, $y$ is equal to $x^w$ for some $w\in (\pi K)^\prime$. Therefore, the algebraic Casson-Whitney number $\afw(K)$ is equal to the minimal number of elements $w_i \in (\pi K)'$ such that the relations $\{[x,x^{w_i}]=1\}$ abelianize $\pi K$. 

These finger move relations are `weaker' than the relations induced by stabilizations, in that every finger move relation is also a stabilization relation. Recall from \autoref{def:algstab} that $\asta(K)$ denotes the minimal number of stabilization relations needed to abelianize the knot group; these relations are of the form $x=y$, where $x$ and $y$ are meridians, or equivalently $[x,w]=1$, where $w\in(\pi K)^\prime$ and $y=x^w$. Thus $\asta(K)$ is the minimal number of elements $w_i \in (\pi K)^\prime$ such the relations $\{[x,w_i]=1\}$ abelianize $\pi K$. Although $x^w$ is not in the commutator subgroup,
$x^w=x[x,w]$, so the finger move relation
$[x,x^w]=1$ is equivalent to the stabilization relation
$[x,[x,w]]=1$,
and we see that $\asta(K)\leq\afw(K)$.

On the other hand, an obvious upper bound for $\afw(K)$ is $\mu(K)-1$, where $\mu(K)$ is the meridional rank of $K$: forcing any single meridian to commute with the rest of a generating set of meridians will force that meridian into the center of the group. Since all knot groups are normally generated by any meridian, this abelianizes the group. We summarize the relationships between these invariants below, which are defined for $n$-knots because we will later refer to the case $n=1$ as well as our usual case $n=2$ (although these invariants are well-defined for all $n\geq 1$ because they only depend on the knot group and the information of a meridian).



\begin{proposition} \label{inequalities}
For any $n$-knot $K$, $$m(K)\leq a(K)\leq \asta(K)\leq \afw(K)\leq\mu(K)-1.$$
\end{proposition}

In \autoref{thm:algadd} we show that the inequality $\asta(K)\leq\afw(K)$ can be strict. In fact, we find infinitely many 2-knots $K$ with $\asta(K)=\ust(K)=1$ and $\afw(K)=2$, enabling us to prove in \autoref{thm:3} that $\ust(K)<\ufw(K)$ for infinitely many 2-knots $K$. The last inequality may also be strict, for the same reason pointed out after \autoref{rem:doublystd}.

\begin{proposition} \label{max}
For $\alpha\in\{a,\asta,\afw\}$ and for $n$-knots $K_1$ and $K_2$,
\[
    \max\{\alpha(K_1),\alpha(K_2)\}
    \leq
    \alpha(K_1\cs K_2)\leq \alpha(K_1)+\alpha(K_2).
\]
\end{proposition}

\begin{proof}
    The proof is the same in all three cases;
    we follow Kanenobu in \cite{kanenobu1996weak} for $\alpha=\asta$. Let $g_1,\dots,g_n$ be a minimal set of relators of the required form (depending on $\alpha$) which abelianize $\pi(K_1\cs K_2)$.
    Let $\phi$ be the surjection
    $\phi \colon \pi(K_1\cs K_2) \twoheadrightarrow \pi K_1$
    which sends all meridians of $K_2$ to the meridian of amalgamation.
    Notice that $\pi K_1/\langle\langle \phi(g_1), \dots, \phi(g_n)\rangle\rangle\cong \mathbb{Z}$ and that each $\phi(g_i)$ is a relator of the required form for computing $\alpha(K_1)$. Therefore, $\alpha(K_1\cs K_2)\geq \alpha(K_1)$. Repeating the argument for $K_2$ obtains the first inequality.
    
    The second is obtained by imposing relations on the group of $K_1 \cs K_2$
    which abelianize $K_1$ and $K_2$ separately.
    Since $\pi(K_1\cs K_2) \cong 
    {(\pi K_1 * \pi K_2)} / {\langle\langle x_1^{-1}x_2\rangle\rangle}$,
    where $x_i$ are meridians of $K_i$,
    these relations abelianize the group of the connected sum.
\end{proof}

As a first application of \autoref{inequalities},
we show that any natural number can occur as the Casson-Whitney number of a 2-knot. We will make use of determinants in the following proposition and in \autoref{thm:algadd}, which we introduce now.

The Alexander module of a 2-knot is the first homology of the infinite cyclic cover, viewed as a $\mathbb{Z}[t^{\pm 1}]$-module. The \bit{determinant} of a 2-knot $K$ is defined in \cite{joseph20190concordance} as the positive generator of the evaluation of the Alexander ideal at $t=-1$, i.e.\ $\Delta (K)|_{-1}:=n$, where $n>0$ is the generator of the principal ideal $\{f(-1):f(t)\in\Delta(K)\}\subseteq\mathbb{Z}$. Equivalently, it is the order of the $\mathbb{Z}$-module induced by setting $t=-1$ in the Alexander module. As with classical knots this is always an odd integer,
and in \cite[Proposition 5.9]{joseph20190concordance} it is shown that even twist-spinning preserves the determinant, while odd twist-spins always have determinant 1. The classical fact that a 1-knot $k$ admits a Fox $p$-coloring for prime $p$ if and only if $p$ divides the classical determinant $|\Delta_k(-1)|$, where $\Delta_k(t)$ is the Alexander polynomial of $k$, carries over without change to this definition of determinant for nonprincipal ideals. A \bit{Fox $p$-coloring} of a 2-knot is a surjection from its knot group onto the dihedral group $D_p\cong \mathbb{Z}_{p}\rtimes \mathbb{Z}_2$, which sends meridians of the 2-knot to reflections. 

\begin{proposition}
    Let $n \in \mathbb{N}$.
    Then there exists a 2-knot $K$ with $\ufw(K)=n$.
\end{proposition}

\begin{proof}
Let $J$ be any 2-knot with $\ufw(J)=1$ and Nakanishi index $m(J)=1$, for instance $J$ could be any even twist-spin of a 2-bridge knot, by \autoref{thm:2bridgecw}: 2-bridge knots have nontrivial determinants, which are preserved by even twist-spinning \cite{joseph20190concordance}. Therefore the Alexander module of $J$ is nontrivial, so it must be cyclic since it is a quotient of the original 2-bridge knot's Alexander module.

Then letting $K=\cs^n J$ obtains the desired result: the Nakanishi index $m(K)=n$, since the Alexander module of $K$ is generated by $n$ elements and surjects onto a vector space of dimension $n$, so by \autoref{inequalities} $\ufw(K)\geq n$. Conversely, $K$ can be unknotted in $n$ pairs of finger and Whitney moves by performing the optimal length one regular homotopy for $J$ on each summand.
\end{proof}

Scharlemann proved that unknotting number one knots are prime, i.e.\ if $K_1$ and $K_2$ are nontrivial classical knots, then the unknotting number of $K_1\cs K_2$ is at least 2 \cite{scharlemann1985unknotting}. Here we prove a special case of the analogous statement for $\ufw$, which works whenever the 2-knots in question have nontrivial determinants, or equivalently whenever their knot groups admit nontrivial Fox colorings.
This reproves the same special case of Scharlemann's theorem for classical knots, via the bound given by \autoref{thm:lower_bound_classical_unknotting}. The technical core of our proof is a Freiheitssatz for one-relator quotients of free products of cyclic groups due to Fine, Howie, and Rosenberger \cite{fine1988freiheitssatz}.

\begin{theorem*}[Fine, Howie, Rosenberger]
    \label{thm:freiheitssatz}
    Suppose $G=\langle a_1,\dots,a_n \mid a_1^{e_1},\dots,a_n^{e_n},R^m\rangle$, where $n\geq 2$, $m\geq 2$, $e_i=0$ or $e_i\geq 2$ for all $i$, and $R(a_1,\dots,a_n)$ is a cyclically reduced word which involves all of $a_1,\dots,a_n$. Then the subgroup of $G$ generated by $a_1,\dots,a_{n-1}$ is isomorphic to 
    $\langle a_1,\dots,a_{n-1} \mid a_1^{e_1},\dots,a_{n-1}^{e_{n-1}}\rangle$.
\end{theorem*}

Their result generalizes the more well-known Freiheitssatz
for one-relator groups,
a classical result in combinatorial group theory characterizing
the torsion in a one-relator group.
It is proved by finding explicit representations of these groups into
$\operatorname{PSL}_2(\mathbb{C})$.

\algadd*

\begin{proof}
    Let $x_1$ and $x_2$ be meridians of $K_1$ and $K_2$ respectively, and form the connected sum so that $x_1$ and $x_2$ become identified in the group of $K_1\cs K_2$, i.e.\
    \[
        \pi(K_1\cs K_2)\cong\frac{\pi K_1 * \pi K_2}{\langle\langle x_1^{-1}x_2\rangle\rangle}.
    \]
    Denote $x$ as the image of these meridians.
    
    The claim to be proved is that for any $w\in \pi(K_1\cs K_2)^\prime$, the relation $[x,x^w]=1$ does not abelianize $\pi(K_1\cs K_2)$, since then $\ufw(K_1\cs K_2)\geq \afw(K_1\cs K_2)\geq2$.
    
    Let $p_1$ and $p_2$ be prime divisors of $\Delta(K_1)|_{-1}$ and $\Delta(K_2)|_{-1}$, respectively.
    Then $K_i$ admits a Fox $p_i$-coloring
    \[
        \phi_i \colon \pi K_i\twoheadrightarrow D_{p_i}
        \cong \mathbb{Z}_{p_i} \rtimes \mathbb{Z}_2
        =\langle z_i,a_i \mid z_i^2=a_i^{p_i}=1,z a_i z=a_i^{-1}\rangle 
    \]
    with $x_i$ mapping to $z_i$, the generator of $\mathbb{Z}_2$.
    The group of the connected sum naturally surjects onto the group
    \begin{align*}
        G & \coloneqq
        \langle z,a_1,a_2
        \mid
        z^2=a_1^{p_1}=a_2^{p_2}=1,z a_1 z=a_1^{-1},z a_2z=a_2^{-1}
        \rangle \\
        & \cong
        (\mathbb{Z}_{p_1}*\mathbb{Z}_{p_2})\rtimes \mathbb{Z}_2,
    \end{align*}
    by mimicking the connected sum operation on the dihedral groups.
    To be explicit, first define
    $\phi_1 * \phi_2 \colon \pi K_1 *\pi K_2\rightarrow D_{p_1}*D_{p_2}$
    in the obvious way. Then $G$ can be constructed from $D_{p_1}*D_{p_2}$ by identifying the images of the meridians:
    \[
        G \cong 
        \frac{D_{p_1}*D_{p_2}}
            {\langle\langle
            z_1^{-1}z_2
            \rangle\rangle}
    \]
    We will show that $ G / \langle\langle\phi([x,x^w])\rangle\rangle $ is not abelian, hence $\pi(K_1\cs K_2)/\langle\langle[x,x^w]\rangle\rangle$ is not abelian either.
    
    Note that $\phi(x)=z$ and $\phi([x,x^w])=[z,z^v]$, where $v=\phi(w)$ is in the commutator subgroup $\mathbb{Z}_{p_1}*\mathbb{Z}_{p_2}$ of $G$. Then $G/\langle\langle[z,z^v]\rangle\rangle$ is the image of the induced homomorphism which we would like to show is nonabelian. We will do this by showing that its commutator subgroup is nontrivial. Let $N=\langle\langle[z,z^v]\rangle\rangle$, the normal closure of $[z,z^v]$ in $G$. As $[z,z^v]$ is a commutator, $N$ is contained in the commutator subgroup $\mathbb{Z}_{p_1}*\mathbb{Z}_{p_2}$ of $G$. The goal now is to show that $(\mathbb{Z}_{p_1}*\mathbb{Z}_{p_2})/N$ is not the trivial group.
    
    
    Note that $[z,z^v]=z (v^{-1}zv) z (v^{-1}zv) = (zv^{-1}zv)^2=[z,v]^2$. It will be convenient to describe $N$ as the normal closure inside of $\mathbb{Z}_{p_1}*\mathbb{Z}_{p_2}$ of some elements of $\mathbb{Z}_{p_1}*\mathbb{Z}_{p_2}$. Denote $g=[z,v]$. Now, $N$ is the normal subgroup generated by all elements of the form $h^{-1}g^2h$, where $h\in G$ is arbitrary. Any $h\in G$ can be written as $z^nc$, where $n=0$ or $1$ and $c\in \mathbb{Z}_{p_1}*\mathbb{Z}_{p_2}$. Then $h^{-1}g^2h=c^{-1}z^ng^2z^nc$. Since $c^{-1}g^2c$ is already in the normal closure of $g^2$ in $\mathbb{Z}_{p_1}*\mathbb{Z}_{p_2}$, it suffices to consider $n=1$, i.e.\ $h=zc$. Notice that $zg^2z=(zgz)^2=(z[z,v]z)^2=(v^{-1}zvz)^2=[v,z]^2=[z,v]^{-2}=(g^2)^{-1}$. Then $c^{-1}zg^2zc=c^{-1}g^{-2}c=(c^{-1}g^2c)^{-1}$, so in fact $N$ is the normal closure in $\mathbb{Z}_{p_1}*\mathbb{Z}_{p_2}$ of just $g^2$. By the Freiheitssatz, $(\mathbb{Z}_{p_1}*\mathbb{Z}_{p_2})/\langle\langle g^2\rangle\rangle$ is nontrivial for any element $g\in\mathbb{Z}_{p_1}*\mathbb{Z}_{p_2}$ (we may assume $g$ is cyclically reduced, since this does not change the isomorphism type of the quotient. If $g$ involves only one of the generators $a_1$ or $a_2$, then clearly the other factor survives in the quotient). 
\end{proof}



\begin{corollary}
    Let $k_1$ and $k_2$ be classical knots with determinants
    $|\Delta_{k_i}(-1)|\neq 1$.
    Then
    \begin{equation*}
        \ufw(\tau^n k_1 \cs \tau^m k_2)\geq 2    
    \end{equation*}
    for any even integers $n,m$. 
\end{corollary}

\begin{corollary}
    \label{addcor}
    Let $K_1$ and $K_2$ be even twist-spins of 2-bridge knots.
    Then $\ufw(K_1\cs K_2)=2$. 
\end{corollary}

\begin{proof}
    Since 2-bridge knots have nontrivial determinants,
    their even twist spins do as well \cite{joseph20190concordance}.
    Then $\ufw(K_1\cs K_2)\geq2$ follows from \autoref{thm:algadd}.
    The reverse inequality follows from \autoref{thm:2bridgecw}
    and the elementary fact
    that $\ufw(K_1\cs K_2)\leq\ufw(K_1)+\ufw(K_2)$.
\end{proof}

It is interesting to note that in the case of 2-bridge knots $k_1,k_2$, the knot group $\pi(\tau^2 k_1 \cs \tau^2 k_2) \cong (\mathbb{Z}_{p_1}*\mathbb{Z}_{p_2})\rtimes \mathbb{Z}$, where $p_i=|\Delta_{k_i}(-1)|$, and that the proof of \autoref{thm:algadd} goes through in that setting without the further quotient to $G$. In fact,
$G$ arises naturally as the group of $\tau^n k_1 \cs \tau^m k_2 \cs \mathbb{R}P^2$,
where $n,m$ are even and $\mathbb{R}P^2$ denotes a standard projective plane.

For odd integers $p,q \in \mathbb{Z}$, let $K_{p,q}$ denote the spin of $T(2,p)\cs T(2,q)$. Miyazaki proved that $\ust(K_{p,q})=1$, whenever $q=p+2,p+4$, or $p+6$,
when $\gcd(p,p+6)=1$ \cite{miyazaki1986relationship}. 
Therefore, $\ust$ fails to be additive in these cases. However, it follows from \autoref{addcor} that $\ufw$ is additive in these cases, and in particular that $\ufw(K_{p,q})=2$. This proves \autoref{thm:3}.

\nonequal*

\subsection{Application to classical unknotting number}
\label{sec:classical_unknotting}

As noted at the start of \autoref{sec:previousalg}, the Nakanishi index, Ma-Qiu index, and algebraic stabilization number are all previously established lower bounds for the classical unknotting number. In this section we point out that the algebraic Casson-Whitney number is also a lower bound for the classical unknotting number, which is sharper than the aforementioned invariants in many cases.

Perhaps the most interesting reason to study $\afw$ as a lower bound for the unknotting number is that the above three invariants all fail to be additive in many simple cases, such as $T(2,p)\cs T(2,q)$ when $p,q$ are coprime \cite{kadoyang}. By \autoref{thm:algadd},
$\afw(T(2,p)\cs T(2,q))=2$ for all (odd) $p,q$. We do not know any cases where $\afw$ fails to be additive on classical knots, although it seems difficult to prove this is always the case. Still, this poses a potentially interesting avenue to study the classical unknotting number, via a lower bound which comes from four dimensional techniques.

Let $k$ be a 1-knot. Remembering that spinning
preserves the knot group (and its meridians),
$\afw(k)=\afw(\tau^0k)$. By \autoref{basic} $\afw(\tau^0k)\leq\ufw(\tau^0k)$, and by \autoref{thm:lower_bound_classical_unknotting}, $\ufw(\tau^0k)\leq u(k)$.
Putting these facts together, we have:

\begin{proposition}
    \label{classical_unknotting_number}
    For any 1-knot $k$, $\afw(k) \leq u(k)$.
\end{proposition}    

As noted in \autoref{sec:algmain}, this reproves a special case of Scharlemann's theorem that unknotting number one knots are prime \cite{scharlemann1985unknotting}. Namely, if $k_1$ and $k_2$ are classical knots with nontrivial determinants, then $u(k_1\cs k_2)\geq 2$.

\section{Strong non-additivity of
\texorpdfstring{$\ust$ and $\ufw$}{Stabilization
and Casson-Whitney number}}
\label{sec:nonadd}

As noted in \autoref{sec:algmain},
Miyazaki was the first to prove that $\ust$ is non-additive.
For certain $p,q$ (see section for precise description)
he showed that $\ust(\tau(T(2,p)\cs T(2,q)))=1$.
As pointed out by Kanenobu \cite{kanenobu1996weak},
the Nakanishi index proves that taking iterated connected sums of
$K=\tau(T(2,p)\cs T(2,q))$ has $\ust(\cs^{n} K)=n$,
while $\ust(\cs^{n} T(2,p))+\ust(\cs^{n} T(2,q)) = 2n$.
This shows the existence of 2-knots $K_1$, $K_2$
with $\ust(K_1)+\ust(K_2)-\ust(K_1\cs K_2)$ arbitrarily large.
In this section we investigate and prove a stronger version of
non-additivity for both the stabilization and Casson-Whitney number.
For notational convenience, throughout the section we use
$\alpha$ to denote either $\asta$ or $\afw$,
and $\nu$ to denote the corresponding $\ust$ or $\ufw$. 

Our geometric study of strong non-additivity is inspired by Kanenobu's work in \cite{kanenobu1996weak} establishing the non-additivity of $\asta$. In particular, for each $n\geq 1$, Kanenobu gave examples of $2$-knots $K_1,\dots,K_n$ with $\asta(K_i)=1$ and $\asta(K_1\cs\cdots\cs K_n) = 1.$

\begin{question}[Kanenobu]
    Is $\ust(K_1\cs\cdots\cs K_n) = 1$ as well?
\end{question}

We generalize Kanenobu's result for $\asta$ and prove a corresponding result for $\afw$. We then prove analogous results for the geometric versions $\ust$ and $\ufw$, answering Kanenobu's question in the affirmative at the expense of a small correction factor. In fact, \autoref{excor} shows that the connected sums $K_1\cs\cdots\cs K_n$ in Kanenobu's original examples have both stabilization number and Casson-Whitney number at most $2$. 

\begin{theorem}\label{thm:algnonadd}
    Let $\alpha \in \{ \asta, \afw\}$.
    Let $K_1,\dots,K_n$ be $2$-knots with $\alpha(K_i) \leq c$ for some $c \in \mathbb{N}$. Suppose that there exist meridians $x_i \in \pi K_i$ and relatively prime integers $j_i\in\mathbb{Z}$ such that each $x_i^{j_i}$ lies in the center $Z(\pi K_i)$ of the knot group of $K_i$.  Then, $\alpha(K_1\cs\cdots \cs K_n)\leq c$. 
\end{theorem}

\begin{proof}

We will prove the case $\alpha=\asta$ and $c=1$ in detail, then point out the changes necessary for the general result.

Since $\alpha(K_i)=1$, there exists an element $w_i\in (\pi K_i)^\prime$ such that $\pi K_i/\langle\langle[x_i,w_i]\rangle\rangle\cong\mathbb{Z}$. Let $K=K_1\cs\cdots\cs K_n$, and let $x=x_i$ be the meridian of amalgamation. We will show that
$\pi K/\langle\langle[x,w_1w_2\cdots w_n]\rangle\rangle\cong\mathbb{Z}$.
For $m\leq n$, let
\begin{align*}
    R_m & = [x,w_1w_2\cdots w_m] \text{ and } \\
    G_m & = \bigslant{ \pi(K_1\cs\cdots\cs K_m) }{ \langle\langle R_m\rangle\rangle }
\end{align*}
Note that $G_1\cong\mathbb{Z}$ by assumption;
we will show that $G_m\cong G_{m-1}$, so that by induction $G_n\cong\mathbb{Z}$. 

Since $j_1$ and $j_2j_3\cdots j_m$ are coprime, there exist integers $s$ and $t$ so that $sj_1+tj_2j_3\cdots j_m=1$. Notice that $x^{sj_1}\in Z(\pi K_1)$ and $x^{sj_1-1}=x^{-tj_2\cdots j_m}\in Z(\pi(K_2\cs\cdots\cs K_m))$. The relation $R_m$ is equivalent to $x=(w_1\cdots w_m)^{-1}x(w_1\cdots w_m)$. Raising both sides to the $sj_1$ we obtain:

\begin{align*}
    x^{sj_1} &= (w_2\cdots w_m)^{-1}w_1^{-1}x^{sj_1}w_1(w_2\cdots w_m)\\
             &= (w_2\cdots w_m)^{-1}x^{sj_1}(w_2\cdots w_m)\\
             &= (w_2\cdots w_m)^{-1}x(w_2\cdots w_m)x^{sj_1-1}
\end{align*}
which is equivalent to $x=(w_2\cdots w_m)^{-1}x(w_2\cdots w_m)$.
We can repeat this procedure until we reach $x=w_m^{-1}xw_m$, or $[x,w_m]=1$,
the relation which abelianizes $\pi K_m$.
Since $w_m$ is in the commutator subgroup of $\pi K_m$,
it is trivial in the abelianization, so the relation $[x, w_{m}] = 1$ abelianizes the subgroup of $\pi(K_1\cs\cdots\cs K_m)$ corresponding to $\pi K_m$, and the induced relation on $\pi(K_1\cs\cdots\cs K_{m-1})$ is $[x,w_1 w_2\cdots w_{m-1}]=1$:

\begin{align*}
    G_m &=
    \bigslant{ \pi(K_1\cs\cdots\cs K_m) }{ \langle\langle [x,w_1w_2\cdots w_m]\rangle\rangle } \\
    & \cong
    \bigslant{ \pi(K_1\cs\cdots\cs K_{m-1}) }{ \langle\langle [x,w_1w_2\cdots w_{m-1}]\rangle\rangle }
    = G_{m-1}.
\end{align*}

Now, if $c>1$, we simply repeat the previous argument $c$ times, making a choice to group the $nc$ assumed relations into $c$ relations, each one the combination of one of the assumed relations from each knot group, as above.

The proof for $\alpha=\afw$ is similar, so we only list the changes here. When $c=1$, each $\pi K_i$ has a finger move relation $[x_i,x_i^{w_i}]=1$ such that $\pi K_i/\langle\langle[x_i,x_i^{w_i}]\rangle\rangle\cong\mathbb{Z}$, for some $w_i\in(\pi K_i)^\prime$. We combine these into one relation: $[x,x^{w_nw_{n-1}\cdots w_1}]=1$, which will abelianize the group of $K_1\cs\cdots\cs K_n$. 

Let $R_m=[x,x^{w_m w_{m-1}\cdots w_1}]$ and let $v_i=w_m w_{m-1}\cdots w_i$,
so e.g.\ $v_1=v_2 w_1$,
and choose $s$ and $t$ as before. Let $G_m=\pi (K_1\cs\cdots \cs K_m)/\langle\langle R_m\rangle\rangle$. As before, we will show by induction that $G_n\cong G_1$, which is infinite cyclic by assumption. The relation which kills $R_m$, $[x,x^{v_1}]=1$,
is equivalent to $x=(x^{v_1})^{-1}xx^{v_1}$.
Raising both sides to the power $sj_1$, we obtain 
\begin{align*}
x^{sj_1} & =v_1^{-1}x^{-1}v_1 x^{sj_1}v_1^{-1}x v_1\\
&=v_1^{-1}x^{-1}v_2 w_1 x^{sj_1} w_1^{-1} v_2^{-1}x v_1\\
&=v_1^{-1}x^{-1}v_2 x^{sj_1} v_2^{-1} x v_1\\
&=v_1^{-1}x^{-1} v_2 x v_2^{-1} x v_2 x^{sj_1-1} w_1\\
&=w_1^{-1}v_2^{-1} x^{-1} v_2 x v_2^{-1} x v_2 x^{-1} w_1 x^{sj_1}
\end{align*}
After canceling the $x^{sj_1}$ terms from both sides, we can further cancel the $w_1$ terms to obtain
$x=v_2^{-1}x^{-1}v_2 x v_2^{-1} x v_2$, or $1=[x,x^{v_2}]$. Repeating this procedure we eventually reach $1=[x,x^{v_m}]=[x,x^{w_m}]$, the relation which abelianizes $\pi K_m$. Thus
\begin{align*}
    G_m & =
    \bigslant{ \pi(K_1\cs\cdots\cs K_m) }{ \langle\langle [x,x^{w_m w_{m-1}\cdots w_1}]\rangle\rangle } \\
    & \cong 
    \bigslant{ \pi(K_1\cs\cdots\cs K_{m-1}) }{ \langle\langle [x,x^{w_{m-1}\cdots w_1}]\rangle\rangle }
    = G_{m-1},
\end{align*}
and by induction $G_n\cong\mathbb{Z}$. The adaptation to $c>1$ is the same as in the previous case.
\end{proof}

\begin{remark}
    \label{examples} 
    There are many nontrivial examples of $2$-knots $K_1, \dots, K_n$ satisfying the hypotheses of \autoref{thm:nonadd}. For instance, the technical condition that the $j^\text{th}$ power of a meridian is central is satisfied by any $j$-twist spun knot \cite{zeeman1965twisting}. Indeed, Kanenobu uses twist-spun knots with coprime twist indices to construct his examples of strong algebraic non-additivity in \cite{kanenobu1996weak}. 
\end{remark}

Recall \autoref{max}, which says that for a pair of 2-knots $K_1,K_2$,
the algebraic lower bounds satisfy
$\max\{\alpha(K_1),\alpha(K_2)\}\leq \alpha(K_1 \cs K_2)\leq \alpha(K_1)+\alpha(K_2)$.
Kanenobu used his nonadditivity result for $\asta$ to prove the following theorem. We note that by \autoref{thm:nonadd}, his original examples work to prove the following corollary for $\alpha=\afw$ as well.

\begin{corollary}[Kanenobu]
    For any positive integers $p_1,\dots,p_n$ and any integer $q$ with $\max\{p_i\}\leq q\leq p_1+\cdots+p_n$, there exist 2-knots $K_1,\dots,K_n$ satisfying:
    \begin{enumerate}
        \item $\asta(K_i)=\afw(K_i)=p_i$ for all $i$, and
        \item $\asta(K_1\cs\cdots\cs K_n)=\afw(K_1\cs\cdots\cs K_n)=q$.
    \end{enumerate}
\end{corollary}

While these examples show that the algebraic Casson-Whitney index $\afw$ is non-additive on general $2$-knot groups, we do not know of any classical knot groups for which this is the case. This is in contrast with the algebraic stabilization number $\asta$, which fails to be additive for classical knots by \cite{miyazaki1986relationship} (see the discussion at the end of \autoref{sec:algebraic_lower_bounds}). Now, to extend these algebraic results on the non-additivity of $\asta$ and $\afw$ to their geometric counterparts $\ust$ and $\ufw$, we first relate these invariants through the following lemma.

\begin{figure}[ht]
    \centering
    \includegraphics[width=\linewidth]{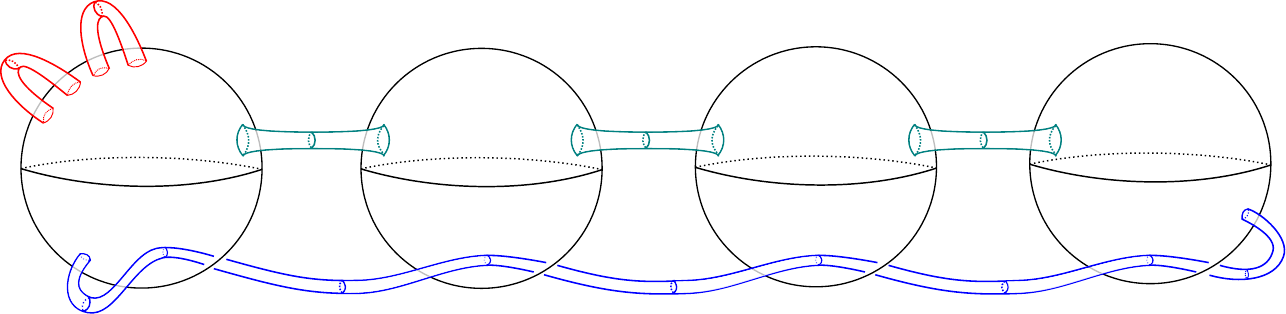}
    \put(-380,85){\small{$K_{1}$}}
    \put(-250,85){\small{$K_{2}$}}
    \put(-120,85){\small{$\cdots$}}
    \put(-10,85){\small{$K_{n}$}}
    \caption{
        A schematic for the proof of \autoref{algtogeolemma}, where $\ust(K_1)=2$ and $\asta(K_1\cs\cdots\cs K_n)=1$. The blue handle abelianizes the group of $K_1\cs\cdots\cs K_n$ and the trivial red handles allow us to inductively unknot each summand.
    }
    \label{fig:nonadditivity}
\end{figure}

\begin{lemma}
    \label{algtogeolemma}
    For $1\leq i \leq n$, let $K_i$ be a 2-knot and let
    $K=K_1\cs\cdots \cs K_n$. If $\ust(K_i) \leq c$ for each $i$,
    then $\ust(K) \leq c + \asta(K)$.
    Similarly, if $\ufw(K_i) \leq c$ for each $i$,
    then $\ufw(K) \leq c + \afw(K)$.
\end{lemma}

\begin{proof}
    We prove the statement
    for the stabilization number
    by induction on the number $n$ of summands. The proof for the Casson-Whitney number is similar. Indeed,
    it will be convenient for the inductive step to prove a slightly stronger statement: For each $n$, any connected sum $K=K_1 \cs \cdots \cs K_n$ can be unknotted by first stabilizing at least $\asta(K)$ times
    to obtain a surface $F$ with $\pi F\cong\mathbb{Z}$,
    and then by stabilizing $c$ times along guiding arcs which are necessarily trivial, since $\pi F$ is cyclic.
    This statement holds in the case $n=1$ since the guiding arcs for the trivial stabilizations can be isotoped in the complement of $F$ to be guiding arcs for a collection of $c$ stabilizations that smoothly unknot $K=K_1$. To proceed with the inductive step, we assume that the statement holds when $n=\ell$ and show that it holds when $n=\ell+1$. 
    

    Let $K=K_1\cs\cdots\cs K_{\ell+1}$.
    Stabilize $K$ $\asta(K)$-times to obtain a surface $F$ with $\pi F\cong \mathbb{Z}$, as before, and let $F^\prime$ be the result of (trivially) stabilizing $F$ an additional $c$-times. We claim that $F^\prime$ is unknotted. Since $\pi F\cong\mathbb{Z}$, the guiding arcs for the $c$ trivial stabilizations are isotopic in the complement of $F$ to guiding arcs for a different set of $c$ stabilizations which unknot the first summand $K_1$. 
    Therefore, $F^\prime$ is the result of trivially stabilizing the surface $F^{\prime\prime}$ $c$-times, where $F^{\prime\prime}$ is the surface obtained from $K_2 \cs \cdots \cs K_{\ell+1}$ by the stabilizations induced by the original $\asta(K)$ stabilizations which produced $F$ from $K$. By the proof of \autoref{max} these stabilizations abelianize $\pi (K_2\cs\cdots\cs K_{\ell+1})$, so $\pi( F^{\prime\prime})\cong\mathbb{Z}$ and by induction the $c$ trivial stabilizations unknot $F^{\prime\prime}$.
    
\end{proof} 

Our first examples of the non-additivity of the stabilization and Casson-Whitney number now follow as a corollary of \autoref{thm:nonadd} and \autoref{algtogeolemma}. 

\begin{corollary}
    \label{excor} 
    For $n\geq1$, consider the $j_i$-twist spins $K_i= \tau^{j_i}k_i$ of classical knots $k_1, \dots, k_n$, where each $k_i$ is either $2$-bridge or has unknotting number one, with pairwise coprime twist indices $j_i\geq 2$. Then,
    \begin{align*}
        & \ust(K_i) = \ufw(K_{i})  = 1 ~ \text{for all $i$,} \\
        & \ust(K_1\cs\cdots\cs K_n)  \leq 2 
        \text{ and }
        \ufw(K_1\cs\cdots\cs K_n)  \leq 2. 
    \end{align*}
\end{corollary}

\begin{proof}
First note that by either
\autoref{thm:lower_bound_classical_unknotting} or \autoref{thm:2bridgecw} (depending on whether the knot $k_i$ is $2$-bridge or unknotting number one), $\ufw(K_i)=1$
for each $i$. So, it just remains to show that $\ust(K_1\cs\cdots\cs K_n) \leq 2$ and $\ufw(K_1\cs\cdots\cs K_n) \leq 2$. This follows from the previous results of this section. In particular, as noted in \autoref{examples} above, the twist spins $K_i$ have $\asta(K_i)=1$ as well as meridians $x_i \in \pi K_i$ such that $x_i^{j_i}\in Z(\pi K_i)$. Therefore, these knots satisfy the hypotheses of \autoref{thm:nonadd}, and so $\asta(K)= \afw(K)=1$ as well.
Now \autoref{algtogeolemma} applies, and we can conclude that both $\ust(K_1\cs\cdots\cs K_n), \ufw(K_1\cs\cdots\cs K_n) \leq 2$, as desired. 
\end{proof} 

Moreover, using a different family of twist spun $2$-knots, we formulate the more general non-additivity result featured in the introduction. 

\nonadd*

\begin{proof}
    Let $\nu \in \{ \ust, \ufw \}$.  For the $i^{\text{th}}$ prime $p_i \in \mathbb{N}$, let $K_i$ be the connected sum of $c$ copies of $\tau^{p_i} T(2,p_i)$, the $p_i$-twist spin of the $(2,p_i)$-torus knot. Since the Alexander module of each summand $\tau^{p_i} T(2,p_i)$ is cyclic, the Nakanishi index $m(K_i)$ of the connected sum is equal to $c$. This matches the upper bound for $\nu$ given by \autoref{thm:2bridgecw},
    and so $\nu(K_i)=c$. Now, each $K_i$ can also be thought of as a single $p_i$-twist spin of the connected sum of $c$ copies of $T(2,p_i)$. Therefore $K=K_1\cs\cdots\cs K_n$ is a connected sum of twist-spun knots with coprime twist indices, and so \autoref{thm:nonadd} applies to show that $\asta(K)=c$. Then by \autoref{algtogeolemma}, $\nu(K)\leq 2c$.
\end{proof}

The proof of the next corollary follows from \autoref{thm:lower_bound_classical_unknotting},
\autoref{thm:2bridgecw}, and \autoref{algtogeolemma}.

\begin{corollary}
    Let $n \in \mathbb{N}$ and let $k_1,\dots,k_n$ be 1-knots,
    each either 2-bridge or with unknotting number one.
    Let $j_1,\dots,j_n$ be coprime integers at least 2 and let $K_i=\tau^{j_i}k_i$. Then $\nu(K_i)=1$ for all $i$ and $\nu(K_1\cs\cdots\cs K_n)\leq 2$,
    where $\nu \in \{\ust, \ufw\}$.
\end{corollary}


\section{Questions}
\label{questions}

\noindent Here we present some questions that remain.

\begin{enumerate}[label*=\arabic*.]

    \item Is $\ust \leq \ufw$, $\ust=\asta$, or $\ufw=\afw$?
    A $2$-knot $K$ with $\ust(K) > \ufw(K)$ or $\ust(K) >\asta(K)$
    would yield a counterexample to the conjecture that smoothly embedded orientable
    surfaces in $S^4$ with knot group $\mathbb{Z}$ are
    smoothly unknotted,\footnote{
        And in fact, it is only known in
        genus $0$ \cite[Thm.\ 11.7A]{freedman-quinn:4-manifolds}
        or if the genus is $\ge 3$ \cite{conway2020embedded}
        that such surfaces are even \emph{topologically} unknotted.
    }
    since in both cases a surface could be obtained whose complement has cyclic fundamental group but which is not smoothly unknotted. 
    On the other hand, a $2$-knot $K$ with $\ufw(K) > \afw(K)$
    would give an immersed $2$-sphere $\Sigma^*$ with
    $\pi_1(S^{4} - \Sigma^*) \cong \mathbb{Z}$
    that is \emph{not} the result of finger moves on the unknot.

    \item Is having a regular homotopy to the unknot
    where the boundaries of the knotted and standard Whitney disks agree
    (as in the proof of \autoref{thm:fusion_upper_bound})
    a characterization of ribbon 2-knots?

    \item Given a $2$-knot $K$ in $S^4$, are Singh's invariants $d_{\textrm{st}}(K)$ and
    $d_{\textrm{sing}}(K)$ from \cite{singh2019distances} ever greater than 1?

    \item Does there exist a 2-knot $K$ such that $\ufw(K)-\ust(K)>1$?
    
    \item Are Casson-Whitney number one 2-knots $K$ ``algebraically prime'', i.e.\ if $K=K_1\cs K_2$, then at least one summand $K_1$ or $K_2$ has knot group $\mathbb{Z}$? 

    \item
    Are pairs of $2$-knots in $S^4$
    always related by an arc-standard regular homotopy?
    Recall that in the proof of \autoref{thm:2}, we prove this for $2$-knots with a length $1$ regular homotopy to the unknot, by starting with a homotopy for which all pairs of knotted and unknotted arcs in the pre-image of the standard immersion have the same endpoints, and then performing `standard braid twists' and isotopies rel endpoints until certain pairs of arcs agree. However, even allowing additional manipulations like `slides' of Whitney disks (as in Figure $4$ of \cite{schneiderman2019homotopy}), such an argument seems to fail for certain initial configurations of Whitney arcs for regular homotopies of higher length, including the one in
    \autoref{fig:boundary_arcs_question}. Thus we ask: can the arcs in \autoref{fig:boundary_arcs_question}
   actually appear as the pre-images of the knotted
    and standard Whitney arcs of a 
   length $2$ regular homotopy from a 2-knot $K$ with $\ufw(K)=2$ to the unknot?
\end{enumerate}

\fig{150}{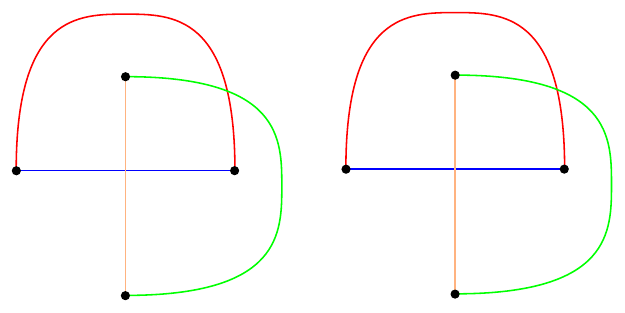}{
    \put(-305,120){\textcolor{red}{$\alpha_1$}}
    \put(-280,75){\textcolor{blue}{$\beta_1$}}
    \put(-235,80){\textcolor{orange}{$\alpha_2$}}
    \put(-160,30){\textcolor{green}{$\beta_2$}}
    \put(-145,120){\textcolor{red}{$\alpha_1^{\ast}$}}
    \put(-120,75){\textcolor{blue}{$\beta_1^{\ast}$}}
    \put(-75,80){\textcolor{orange}{$\alpha_2^{\ast}$}}
    \put(0,30){\textcolor{green}{$\beta_2^{\ast}$}}
    \caption{
        Can the red \& orange above appear as the standard Whitney arcs and the blue \& green as the
        knotted Whitney arcs in a regular homotopy
        to the unknot?
    }
    \label{fig:boundary_arcs_question}
}

\printbibliography

\end{document}

%% file: commands.tex
\newcommand{\ZZ}{\mathbb{Z}}

\newcommand{\fig}[3]{\begin{figure}\includegraphics[height=#1pt]{#2}#3\end{figure}}

\newcommand{\ufw}{u_{\textup{cw}}} 
\newcommand{\ust}{u_{\textup{st}}}
\newcommand{\afw}{a_{\textup{cw}}} 
\newcommand{\asta}{a_{\textup{st}}}

\newcommand{\fus}{\operatorname{fus}} 

\newcommand{\dsing}{d_{\textup{sing}}} 
